\documentclass[reqno]{amsart}

\usepackage{amsthm, amsmath, amsfonts, amscd, amssymb}
\usepackage{stmaryrd}
\usepackage{mathrsfs}
\usepackage[colorlinks,linkcolor=red,citecolor=green,anchorcolor=blue]{hyperref}
\usepackage{color}
\usepackage{tikz}
\usetikzlibrary{arrows} \usetikzlibrary{patterns}
\usepackage[all]{xy}

\newtheorem{theorem}{Theorem}[section]
\newtheorem{lemma}[theorem]{Lemma}
\newtheorem{prop}[theorem]{Proposition}
\newtheorem{defn-prop}[theorem]{Definition-Proposition}
\newtheorem{coro}[theorem]{Corollary}

\newtheorem{defn}[theorem]{Definition}

\theoremstyle{definition}
\newtheorem{example}[theorem]{Example}

\newtheorem{remark}[theorem]{Remark}

\numberwithin{equation}{section}

\def \fk{\mathfrak}
\def \msf{\mathsf}

\def \mc{\mathcal}
\def \scr{\mathscr}
\def \inv{^{-1}}
\def \0{\infty}

\def \n{\noindent}
\def \cplane{\mathbb{C}}
\def \integer{\mathbb{Z}}
\def \rone{\mathbb R}
\def \ration{\mathbb{Q}}
\def \qq{\quad}
\def \p{\partial}

\def \red{\textcolor{red}}

\def \rto{\rightarrow}

\def \hrto{\hookrightarrow}

\newcommand{\Conj}[1]{[ #1]}

\def \ord{\mathrm{ord}}
\def \orb{{\mathrm{orb}}}

\newcommand{\inverse}{^{-1}}
\newcommand{\HH}{{\scr H}}

\def \one{{\mathbf 1}}

\def \g{{\vec g}}
\def \h{{\vec h}}
\def \ee{{e_G}}
\def \M{{\scr M}}
\def \bM{{\overline\M}}
\def \rr{{\vec r}}
\def \z{{\vec z}}

\allowdisplaybreaks

\begin{document}

\title[Equivariant commutative stringy cohomology rings]{Equivariant commutative stringy cohomology rings on almost complex manifolds}

\author{Bohui Chen}
\address{School of Mathematics and Yangtze Center of Mathematics,
Sichuan University, Chengdu 610065, China}
\email{bohui@cs.wisc.edu}

\author{Cheng-Yong Du}
\address{School of Mathematics, Sichuan Normal University,
Chengdu 610068, China}
\email{cyd9966@hotmail.com}

\author{Tiyao Li}
\address{School of Mathematics, Chongqing Normal University,
Chongqing 401331, China}
\email{litiyao@sina.com}

%
\thanks{The first author was supported by National Natural Science Foundation of China (No. 11431001 and No. 11726607). The second author was supported by National Natural Science Foundation of China (No. 11501393).}

\begin{abstract}
In this paper, motivated by Chen--Ruan's stringy orbifold theory on almost complex orbifolds, we construct a new cohomology ring $\HH^\ast_{G,cs}(X)$ for an
equivariant almost complex pair $(X,G)$,
where $X$ is a compact connected almost complex manifold,
$G$ is a connected compact Lie group
which acts on $X$ and preserves
the almost complex structure.
\end{abstract}

\maketitle


\section{Introduction}

Stringy cohomology theory on orbifolds was motivated from physics and
it is discovered in mathematics by Chen--Ruan as the Chen--Ruan cohomology
ring $H^*_{CR}(\mathsf M)$ (\cite{CR04}) in 2004, where $\mathsf M$ is an almost complex orbifold.
For symplectic orbifolds, this cohomology ring is the classical limit of the orbifold quantum cohomology ring (cf. \cite{CR02}).   The main ingredient
 in Chen--Ruan's ring structure is  the obstruction bundle whose
  fiber  is interpreted as the cokernel of certain elliptic operator.
  The computation of the obstruction bundle was later discovered  by  S. Hu and the first author \cite{CH06} for abelian case, and by Hu--Wang for general
  cases in \cite{HW13}.  Their computations are crucial for our construction in this paper.

The ordinary equivariant cohomology theory could not detect the information of subspace with finite stabilizer.
   Motivated by Chen--Ruan's theory, it is expected that if {\em there exists a  Chen--Ruan type equivariant theory when $X$ is a compact (almost) complex manifold which
admits a connected compact Lie group $G$ action.}  We may call such a theory as an equivariant stringy cohomology theory for the pair $(X,G)$.
When $G$ is a finite group or $G=T$ is a torus,
the problem can be perfectly solved. This is due to the fact that $[\fk g,\fk g]=0$,
where $\fk g$ is the Lie algebra of $G$.  However, when $G$ is a connected non-commutative compact Lie group, the problem is still open. In this paper, our goal is to construct such a theory.

We list the works related to equivariant stringy cohomology theory for the pair $(X,G)$.
\begin{enumerate}
\item  When $G$ is finite, this is essentially the Chen--Ruan's theory. Fantechi--G\"ottsche (\cite{FG03}) and
Jarvis--Kaufmann--Kimura (\cite{JKK07})
studied the pair $(X,G)$, where $G$ is a finite group.  They show that the theory constructed
is same as the Chen--Ruan's.
\item
 When $G$ is an abelian group, the equivariant stringy cohomology theory can be constructed by
 Goldin--Holm--Knutson (\cite{GHK07}) and it is applied to study the
Kirwan morphisms for stringy cohomologies. The idea appeared also in \cite{CH06}.  Later, Becerra--Uribe \cite{BU09} extended the constructions in \cite{CH06,JKK07} to (twisted) K-theory for both cases that $G$ is either abelian or finite.
 \item When  $G$ is a non-abelian connected compact Lie group,
there is an attempt made by Edidin--Jarvis--Kimura \cite{EJK10}. They considered the case that $G$ is a reducible algebraic group, acting on an algebraic space $X$, and $[X/G]$ is an orbifold (hence the action has finite stabilizer). They used the formal bundle $TX-\fk g$ to constructed the obstruction bundle hence is essentially the theory of quotient orbifold $[X/G]$. Since such construction requires that $\fk g\hookrightarrow TX$ on $X$, $X$ is very restrictive, e.g, $X$ can not be compact.
\end{enumerate}
There are some other works may  relate to this problem and may shed light on it. Let $(X,\omega, G,\mu)$ be a symplectic manifold with a hamiltonian $G$-action, and $\mu$ be the moment map. Let $\msf M=X\sslash G$ be the symplectic reduction. It is known that the moduli spaces of symplectic vortices may give a Gromov--Witten type theory, which people call the Hamiltonian Gromov--Witten theory (cf. \cite{CGS00,M00,CGMS02,M03,GS05}).
Recently, Chen--Wang--Wang's work \cite{CWW14,CWW18a} says that $L^2$-Hamiltonian Gromov--Witten theory implies a Gromov--Witten type theory on the orbifold $\mathsf M$, and they also
introduce a new equation, so called  the augmented symplectic vortex equation, and use it to
quantize the equivariant stringy cohomology  for the pair $(X,G)$ when $G$ is {\em abelian} (cf. \cite {CWW18b}). However, we do not know if there are some way to overcome the nonabelian case on this direction. Nevertheless, the work of symplectic vortices inspires us
to consider certain moduli space of degree zero. The construction in this paper is based on this.
Let $(X,G)$ be a compact  almost complex manifold pair, i.e, $X$ is a compact almost complex manifold and  $G$ is a connected compact Lie group
acting on $X$ and preserving the almost complex structure.
We construct a cohomology ring $\HH_{G,cs}^*(X)$ for the pair which we call the {\em equivariant commutative stringy cohomology ring}.  We make some remarks on the technical issues in our construction.
\begin{enumerate}
\item About the obstruction bundle:
 our construction of obstruction bundle uses $TX$ which is  different from the construction in \cite{EJK10}, where they used $TX-\fk g$, hence our construction works for any compact $X$ without further assumption.
 \item About the term ``commutative'' (cf. Definition \ref{D conjugate-class}): unlike the Chen--Ruan's construction, in the definition of equivariant commutative stringy product in\ \eqref{E def-cs-X-product} we require elements in tuple $(g_1,\ldots, g_m)$ are commutative to avoid infinite summations. A similar consideration
 also appears in the works of A. Adem and his collaborators (cf. \cite{ACTG12,AG12,ACG13,AG15}).
 \end{enumerate}

This paper is organized as follow. In \S \ref{S 2} we define the equivariant commutative stringy cohomology group. In \S \ref{S 3} we construct the obstruction bundle and the equivariant commutative stringy cohomology ring. Finally in \S \ref{S 4} we study the relation between equivariant commutative stringy cohomology ring and symplectic reduction. As an application, we consider an orbifold $\sf M$ which is a symplectic reduction of a hamiltonian system $(X,\omega,G,\mu)$ with $G$ being connected and compact; the equivariant commutative stringy cohomology ring of the pair $(X,G)$ induces a new stringy cohomology ring of the Chen--Ruan cohomology group of $\sf M$. The appendix \ref{S appendix} provides a technical result of the existence of equivariant volume form over certain bundles associated to a connected compact Lie group.

\subsection*{Acknowledgement}
The authors thank Rui Wang and Yu Wang for useful discussions.

\section{Equivariant commutative stringy cohomology group \\ on
almost complex manifolds}\label{S 2}

\subsection{Inertia manifolds}

Let $G$ be a connected compact Lie group acting on an almost complex manifold $(X,J)$ and preserving the almost complex structure $J$. For an element $g\in G$, denote by $\Conj g$ the conjugate class of $g$ in $G$. Denote the unit element by $\mathbf 1 \in G$. Let $G_{\sf f}$ denote the subspace of all finite order elements in $G$.

\begin{defn}\label{D conjugate-class}
For $m\in\integer_{\geq 1}$, we set
\[
G_{\sf f}^m:=\left\{\g=(g_1,\ldots,g_m)\in \underbrace{G_{\sf f}\times\cdots\times G_{\sf f}}_{m} \left|\substack{
g_1,\ldots,g_m \in G_{\sf f}\cap T\text{ for}\\ \text{a maximal torus }T \text{ of }G \text{ and}\\
\text{the subgroup }\langle\g\rangle=\langle g_1,\ldots,g_m\rangle \\ \text{ generated by }\g \text{ is finite}}
\right.\right\}
\]
and
\[
\Conj{G_{\sf f}^m}:=\{\Conj \g\mid \g\in G_{\sf f}^m\}
\]
to be the set of conjugate classes of $m$-tuples in $G_{\sf f}^m$ under the
diagonal conjugation action of $G$, i.e.
\[
\Conj\g=\Conj{g_1,\ldots,g_m}
:=\{h\g h\inverse:=(hg_1h\inverse,\ldots,hg_mh\inverse) \mid h\in G\}.
\]
When $m=1$, $G^1_{\sf f}=G_{\sf f}$, so we omit the superscript.
\end{defn}

We endow each conjugate class $[\g]$ the subspace topology of $G_{\sf f}^m$. For a $g\in G$ denote by $C_G(g)$ the centralizer of $g$ in $G$, or simply by $C(g)$. We also set
\[
C(\g):=\bigcap_{i=1}^m C(g_i)
\]
for a $\g=(g_1,\ldots, g_m)$. Then we have a diffeomorphism $[\g]\cong G/C(\g)$.

Denote the $G$-action on $X$ by $gx$ for a $g\in G$ and an $x\in X$.

\begin{defn}\label{def inertial-mfld}
For each $[\g]\in\Conj{G_{\sf f}^m}$, we set
\begin{align*}
X_{[\g]}:=\{(x,\g)\mid\g\in[\g],\g\cdot x=(g_1x,\ldots,g_mx)=(x,\ldots, x)\}
\end{align*}
which is viewed as a submanifold of $X\times G_{\sf f}^m$. We call $X_{[\g]}$ an {\bf $m$-sector} of $(X,G)$. We call $X_{\Conj g}$ a {\bf twisted sector} of $(X,G)$ when $m=1$ and $\Conj \g\neq \Conj \one$, and $X_{\Conj \one}=X\times\{\one\}$ the {\bf non-twisted sector}.

For each $m\in\integer_{\geq 1}$ we set
\begin{align*}
I_G^m(X):=\bigsqcup_{\Conj\g\in \Conj {G_{\sf f}^m}} X_{\Conj\g}
\end{align*}
to be the disjoint union of $X_{[\g]}$ over $[\g]\in \Conj{G_{\sf f}^m}$. We call $I_G^m(X)$ the {\bf $m$-inertia manifold} of $(X,G)$. When $m=1$ we call $I^1_G(X)$ the {\bf inertia manifold} and denote it simply by $I_G(X)$.

$G$ acts on $I_G^m(X)$ by $h\cdot (x,\g)=(hx,h\g h\inverse)$. It preserves every $X_{[\g]}$.
\end{defn}

For a $\Conj \g\in \Conj{G_{\sf f}^m}$, fix a representative $\g=(g_1,\ldots,g_m)$, let
\[
X^{\g}:=\bigcap_{i=1}^m X^{g_i}
\]
be the set of fixed loci of the subgroup $\langle\g\rangle$ of $G$ that are generated by $g_1,\ldots,g_m$, where $X^{g_i}$ is the fixed loci of $g_i$-action on $X$. There is a $C(\g)$-action on the product space $X^{\g}\times G$ given by
\[
g\cdot (x,k)=(g x, k g\inv).
\]
Let $X^\g\times_{C(\g)}G$ be the quotient space of this action. Then $G$ acts on $X^\g\times_{C(\g)}G$ by multiplying from the left to the second factor $G$.

\begin{lemma}\label{lem iso-difrnt-rep-of-twistd-sectors}
There is a $G$-equivariant diffeomorphism
\begin{align*}
\phi:X_{\Conj \g}\rto& X^{\g}\times_{C(\g)}G,\qq
(x,k\g k\inv)\mapsto [k\inv\cdot x,k]
\end{align*}
with inverse map given by $[x,h]\mapsto(h\cdot x,h\g h\inv)$.
\end{lemma}

There are natural maps between $m$-sectors $I_G^m(X)$, the inertia manifold $I_GX$ and the ambient manifold $X$. We list them as follow:
\begin{defn}\label{def natural-maps}
\begin{enumerate}
\item For $m\geq 1$, $e:I_G^m(X)\rto X, \, (x,\g)\mapsto x$.
\item For $m\geq 2$ and $1\leq i\leq m$,
       \[
       e_i: I_G^m(X)\rto I_G(X),\qq
       (x,(g_1,\ldots,g_m))\mapsto(x,g_i),
       \]
       and
       \[
       e_\0:I_G^m(X)\rto I_G(X),\qq (x,\g)\mapsto(x,\g_\0),
       \]
       where $\g_\0:= g_1\cdot\ldots\cdot g_m$, and
       \[
       e_0:I_G^m(X)\rto I_G(X),\qq (x,\g)\mapsto(x,{\g_\0}\inv),
       \]
\item For $m\geq 2$ and $1\leq l\leq m$,
       \[
       e_{i_1,\ldots,i_l}: I_G^m(X)\rto I_G^l(X),\qq
       (x,(g_1,\ldots,g_m))\mapsto(x,(g_{i_1},\ldots,g_{i_l})).
       \]
       When $l=1$, we get those $e_i$ in (2).
\end{enumerate}
All these maps are $G$-equivariant.
\end{defn}

\subsection{Equivariant commutative stringy cohomology group}

We first describe a degree shifting. For each point $(x,g)\in I_G(X)$, $T_xX$ decomposes into eigen-spaces of $g$-action
\begin{align}\label{E decompose-TX-eigen}
T_xX=\bigoplus_{0\leq j\leq \text{ord}(g)-1} T_{x,g,j},
\end{align}
where $\text{ord}(g)$ is the order of $g$, and $g$ acts on $T_{x,g,j}$ by multiplying
\begin{align}\label{E g-action-weight-on-eigen}
e^{2\pi \sqrt{-1}\frac{j}{\text{ord}(g)}}.
\end{align}

\begin{defn}
For each (twisted) sector $X_{[g]}$, the degree shifting is
\[
\iota([g]):=\sum_{0\leq j\leq \mathrm{ord}(g)-1}
\frac{j}{\mathrm{ord}(g)} \cdot\dim_\cplane T_{x,g,j}.
\]
Over each connected component of $X_{[g]}$, this is a constant. 
\end{defn}

Now we define the equivariant commutative stringy cohomology group of $(X,G)$.
\begin{defn}\label{D HH(X,G)}
The equivariant commutative stringy cohomology group of $(X,G)$ is
\[
\HH^\ast_{G,cs}(X):=\bigoplus_{[g]\in[G_{\sf f}]} H^{*-2\iota([g])}_G(X_{[g]}).
\]
\end{defn}
In this paper, we take $\rone$ as the coefficient field of (equivariant) cohomology group. One can also take $\ration$ or $\cplane$. In the following we abbreviate ``equivariant commutative string cohomology'' into ``ECS-cohomology'' for simplicity.

\begin{remark}
One should note that, up to now, all the definitions and constructions work without the assumption on the commutativity of the tuples $\g=(g_1,\ldots, g_m)$, i.e. even if $g_1,\ldots, g_m$ do not lie in a maximal torus of $G$, all the definitions and constructions above work.
\end{remark}

\section{Equivariant commutative stringy cohomology rings\\ on
almost complex manifolds}\label{S 3}

In this section we construct a ring structure over the ECS-cohomology group.

\subsection{Moduli space of degree zero maps}

In this subsection we describe a moduli space of certain degree zero maps. Take a $\Conj\g\in\Conj{G^m_{\sf f}}$ and a representative $\g=(g_1,\ldots,g_m)$. In this subsection we always assume that $m\geq 2$. Set $g_0:=\g\inv_\0=(g_1\cdot\ldots\cdot g_m)\inv$. We have an $(m+1)$-tuple of positive integers $\rr=(r_0,r_1,\ldots,r_m)$ with $r_i=\ord(g_i)$, which depends only on $\Conj\g$ not on the representatives.

There is an orbifold sphere $S^2_\orb$ with $(m+1)$ orbifold points $\z=(z_0,z_1,\cdots,z_m)$ such that the isotropy group at $z_i$ is $\integer_{r_i}$, the cyclic group of order $r_i$, for $0\leq i\leq m$. Denote this orbifold sphere by $\msf S_\rr:=(S^2_\orb,\z,\rr)$. $\msf S_\rr$ depends only on $\Conj\g$. Its orbifold fundamental group has a presentation given by (cf. \cite{ALR07,CH06,CR04})
\[
\pi^\orb_1(\msf S_\rr)=\{\lambda_0,\lambda_1,\ldots,\lambda_m|
\lambda_0\cdot\ldots\cdot\lambda_m=1, \text{ and }\lambda_i^{r_i}=1
\}.
\]

There is a group homomorphism
\[
\psi:\pi^\orb_1(\msf S_\rr)\rto G
\]
given by $\psi(\lambda_i)=g_i$, $i=0,1,\ldots,m$. Let $\vec\lambda=(\lambda_1,\ldots,\lambda_m)$. Then we could denote $\psi$ by $\psi(\vec\lambda)=\g$. Denote the kernel by $N$, which also depends only on $\Conj\g$. The image of $\psi$ is the subgroup $\langle \g\rangle=\langle g_1,\ldots,g_m\rangle$ of $G$. By Definition \ref{D conjugate-class}, $\langle\g\rangle$ is a finite group. Therefore there is a smooth closed Riemann surface $\Sigma$, and an orbifold covering (cf. \cite{ALR07,CH06,CR04})
\[
\pi:\Sigma\rto \msf S_\rr
\]
with deck transformation group isomorphic to $\langle \g\rangle$ and $\pi_1(\Sigma)=N$. Therefore $\pi^\orb_1(\msf S_\rr)$ acts on $\Sigma$ via the homomorphism $\pi^\orb_1(\msf S_\rr)\rto \langle\g\rangle$, and $N$ acts on $\Sigma$ trivially. Then we have
\[
\frac{\Sigma}{\pi^\orb_1(\msf S_\rr)/N}=\frac{\Sigma}{\langle \g\rangle} =\msf S_\rr.
\]
One can change the representative $\g$ of $\Conj\g$, hence $\langle \g\rangle$, without changing $N$ and $\Sigma$.

Consider the following space
\[
\widetilde \M_{\Conj \g}:=
\left\{(f,\psi):(\Sigma,j;\pi^\orb_1(\msf S_\rr)) \rto (X,J;G)
\left| \substack{ f(\Sigma)=x\in X,\\\psi(\vec \lambda)=\g\,\in \Conj\g,
\\ f \text{ is }J\text{-}j \text{ holomorphic }\\ \text{and equivariant } w.r.t.\, \psi} \right. \right\}.
\]
The group
$\mbox{Aut}(\Sigma,\pi^\orb_1(\msf S_\rr))$ of holomorphic automorphisms of $\Sigma$ that commute
with the $\pi^\orb_1(\msf S_\rr)$-action acts on $\widetilde \M_{\Conj \g}$. Then we get the following moduli space
\[
\M_{\Conj\g}:=\widetilde\M_{\Conj\g}\Big/ \mbox{Aut}(\Sigma,\pi^\orb_1(\msf S_\rr)).
\]
One could view this as a combination of the moduli space of certain orbifold pseudo-holomorphic curves and the moduli space of certain symplectic vortices. On the other hand, $G$ acts on $\widetilde\M_{\Conj\g}$ by transforming the images of $f$ and conjugating the images of $\psi$. This $G$-action commutes with the $\mbox{Aut}(\Sigma,\pi^\orb_1(\msf S_\rr))$-action. Hence $G$ acts on $\M_{\Conj\g}$.

By allowing $[\g]$ varying in the whole $[G^m_{\sf f}]$ we get
a moduli space
\[
\M_m=\bigsqcup_{\Conj\g\in\Conj{G^m_{\sf f}}}\M_{\Conj\g}.
\]
We have a natural map $\pi:\M_{\Conj\g}\rto\mc M_{0,m+1}$ to the moduli space of $(m+1)$-marked smooth closed genus zero Riemann surfaces, by mapping an element $[(f,\psi):(\Sigma,j,\pi^\orb_1(\msf S_\rr))\rto (X,J;G)]$ to the equivalent class of the coarse space of the orbifold sphere $\msf S_\rr=\Sigma/(\pi^\orb_1(\msf S_\rr)/N)=\Sigma/\langle \g\rangle$. As the orbifold case (cf. \cite{CR04}),
\begin{prop}
We have $\M_{\Conj\g}\cong \mc M_{0,m+1}\times X_{\Conj\g}$, hence
$\M_m\cong \mc M_{0,m+1}\times I^m_G(X)$.
Moreover, $\M_{\Conj\g}$ can be compactified into $\bM_{\Conj\g}$ with
\[
\bM_{\Conj\g}\cong \overline{\mc M}_{0,m+1}\times X_{\Conj\g},\qq
\bM_m\cong \overline{\mc M}_{0,m+1}\times I_G^m(X).
\]
\end{prop}

\begin{lemma}\label{lem maps-on-moduli-space}
We have obvious evaluation maps $ev_i:\bM_m\rto I_G(X)$ for $i=0,1,\ldots,m$ and $\0$
which factor through the following composition
\[
\xymatrix{
\bM_m\cong \overline{\mc M}_{0,m+1}\times I_G^m(X)
\ar[r]^-{\mathrm{proj}}&
I_G^m(X)
\ar[r]^-{e_i} &
I_G(X).}
\]
On $\M_m$, for $i=0,1,\ldots, m$, $ev_i$ is
\[
[(f,\psi):(\Sigma,j,\pi^\orb_1(\msf S_\rr))\rto(X,J;G)]
\mapsto (f(\Sigma),g_i=\psi(\lambda_i))
\]
and $ev_\0$ is
\[
[(f,\psi):(\Sigma,j,\pi^\orb_1(\msf S_\rr))\rto(X,J;G)]
\mapsto (f(\Sigma),g_\0=\psi(\lambda_0\inv)).
\]
Restrict to each component we get $ev_i:\bM_{\Conj\g}\rto X_{\Conj{g_i}}$, for $i=0,1,\ldots,m$ and $\0$.
\end{lemma}

\subsection{Obstruction bundle}

In this subsection we define an obstruction bundle $\scr O_m$
for each moduli space $\M_m$. Consider a component $\M_{\Conj\g}=\mc M_{0,m+1}\times X_{\Conj\g}$, and an element
\begin{align}\label{E an-element-in-M-[g]}
[(f,\psi):(\Sigma,j,\pi^\orb_1(\msf S_\rr))\rto(X,G)]\in \M_{\Conj\g}
\end{align}
with image $f(\Sigma)=x\in X$ and $\psi(\vec \lambda)=\g=(g_1,\ldots, g_m)$. Then $g_i\cdot x=x$ and there is a $\langle \g\rangle$-equivariant elliptic complex
\[
\bar\p:\Omega^{0,0}(\Sigma, f^*T_xX)\rto\Omega^{0,1}(\Sigma, f^*T_xX)
\]
over $\Sigma$. Consider its $\langle \g\rangle$-invariant part, which is also an elliptic complex over $\Sigma$. Then we get a space
\[
H^{0,1}(\Sigma,f^*T_xX)^{\langle \g\rangle}.
\]
This forms a bundle $\scr O_{\Conj\g}$ over $\M_{\Conj\g}$,
which is $G$-equivariant with respect to the induced $G$-action on $TX$. Denote the disjoint union of all $\scr O_{\Conj\g}$ by $\scr O_m$.

\begin{defn}\label{def obs-bundle}
We call $\scr O_m$ the obstruction bundle over $\M_m$.
\end{defn}

We next give another description of $\scr O_m$, which is similar to the K-theory description of the Chen--Ruan obstruction bundle obtained by Chen--Hu \cite{CH06} and Hu--Wang \cite{HW13}, see also \cite{FG03,GHK07,JKK07,BU09,DL17}. Consider the element \eqref{E an-element-in-M-[g]}. As above, suppose that $\psi(\vec\lambda)=\g=(g_1,\ldots, g_m)$ and $f(\Sigma)=x$. As in previous subsection, we set $g_0=\g\inv_\0=(g_1\ldots g_m)\inv$. Let $C\langle\g\rangle$ denote the center of the group $\langle\g\rangle$. By Definition \ref{D conjugate-class}, $\langle\g\rangle$ is finite and $C\langle\g\rangle=\langle\g\rangle$.

The tangent space $T_xX$ is a complex representation of $\langle\g\rangle$. We decompose $T_xX$ into direct sum of $\langle\g\rangle$-irreducible representations
\begin{align}\label{E decompose-TX-irreducible}
T_xX=\bigoplus_{\lambda\in\widehat{\langle\g\rangle}} T_{x,\lambda}.
\end{align}
Note that for $0\leq i\leq m$, $g_i$ acts on $T_{x,\lambda}$ for all $\lambda\in \widehat{\langle\g\rangle}$.
Since $\text{ord}(g_i)<\infty$, $T_{x,\lambda}$ decompose into eigen-spaces of $g_i$. Then we see that each eigen-space of $g_i$ is also a representation of $\langle\g\rangle$. Therefore by the irreducibility of $T_{x,\lambda}$, it is an eigen-space of $g_i$. So for $i=0,1,\ldots, m$, $g_i$ acts on each $T_{x,\lambda}$ by multiplying $e^{2\pi w_{\lambda,i} \sqrt{-1}}$ for a number $w_{\lambda,i}\in[0,1)\cap\ration$, called the {\em weight}\footnote{Comparing with the eigen-space decomposition of $T_xX$ under the $g$ action in \eqref{E decompose-TX-eigen} and the action weight in \eqref{E g-action-weight-on-eigen}.}. Since $g_0g_1\ldots g_m=\one$, one has
\begin{align*}
w_{\lambda,0}+w_{\lambda,1}+\ldots+w_{\lambda,m}=0, \,\mathrm{or} \, 1,\,\mathrm{or} \, \ldots,\,\mathrm{or} \,  m.
\end{align*}
Define $m+1$ formal vector spaces
\begin{align*}
\mc S_{g_i,x}:=\bigoplus_{\lambda\in\widehat{\langle\g\rangle}} w_{\lambda,i}T_{x,\lambda},\qq \mathrm{for}\qq i=0,\ldots, m.
\end{align*}
On the other hand, the normal spaces at $x$ of $X^{g_i}$ and $X^\g$ in $X$ are
\begin{align}\label{E def-N-g-x}
\begin{split}
N_{g_i,x}&=\bigoplus_{\lambda\in\widehat{\langle\g\rangle},\; w_{\lambda,i}>0} T_{x,\lambda},\qq \mathrm{for}\qq  i=0,1,\ldots,m,\\
\mathrm{and}\qq N_{\g,x}&=\bigoplus_{\lambda\in\widehat{\langle\g\rangle},\; \sum_{i=0}^m w_{\lambda,i}\geq 1} T_{x,\lambda}
\end{split}
\end{align}
respectively. It is direct to see that
\begin{align}\label{E S+S=N}
\mc S_{g_i,x}\oplus \mc S_{g_i\inv,x}=N_{g_i,x}.
\end{align}
Moreover,
\begin{align}\label{E iota-g=rank-S-g}
\iota([g_i])=\text{rank}_\cplane\, \mc S_{g_i,x}.
\end{align}
\begin{lemma}\label{L de-rham-obs}
The fiber of $\scr O_m$ over a point $[f,\psi]$ as in \eqref{E an-element-in-M-[g]} is
\[
\bigoplus_{i=0}^m\mc S_{g_i,x}\ominus N_{\g,x}=
\bigoplus_{\lambda\in \widehat{\langle\g\rangle},\;
\sum_{i=0}^m w_{\lambda,i}\geq 2}
\Big(\sum_{i=0}^m w_{\lambda,i}-1\Big)\cdot T_{x,\lambda}.
\]
These vector spaces form a $G$-bundle over $\M_m$, and is isomorphic to $\scr O_m$ as a $G$-bundle.
\end{lemma}
\begin{proof}
After replacing the bundle $T\mc G^0$ by $TX$ in the proof of \cite[Theorem 3.2]{HW13} we get this lemma.
\end{proof}

\subsection{Equivariant commutative stringy product}
In this subsection we will use the obstruction bundle $\scr O_2$ over $\scr M_2=I_G^2(X)$ to define the equivariant commutative stringy product (ECS-product) $\star_{cs}$ over $\HH^\ast_{G,cs}(X)$.

All maps in Definition \ref{def natural-maps} decompose naturally into compositions of $G$-equivariant embeddings and $G$-equivariant fiber bundle projections. For example, take a $[\g=(g_1,\ldots,g_m)]\in[G_{\sf f}^m]$. For $1\leq i_1<\ldots <i_k\leq m$, set $\g_{i_1,\ldots,i_k}=(g_{i_1},\ldots, g_{i_k})$, we have the following diagram
\begin{align}\label{D}
\xymatrix{
X^\g\times_{C(\g)}G \ar@{^{(}->}[rr]^-{\tilde e_{i_1,\ldots, i_k}} & &
X^{\g_{i_1,\ldots,i_k}}\times_{C(\g)} G\ar[rr]^-{p_{i_1,\ldots,i_k}}\ar[d]& & X^{\g_{i_1,\ldots,i_k}} \times_{C(\g_{i_1,\ldots,i_k})}G\ar[d]\\
&&G/C(\g)\ar[rr] && G/C(\g_{i_1,\ldots,i_k}),}
\end{align}
where
\begin{itemize}
\item $\tilde e_{i_1,\ldots,i_k}$ is obtained from the $C(\g)$-equivariant inclusion $X^\g\hrto X^{\g_{i_1,\ldots,i_k}}$, hence is a $G$-equivariant embedding,
\item $p_{i_1,\ldots,i_k}:X^{\g_{i_1,\ldots,i_k}}\times_{C(\g)} G \rto X^{\g_{i_1,\ldots,i_k}} \times_{C(\g_{i_1,\ldots,i_k})}G$ is the pull back bundle of $G/C(\g)\rto G/C(\g_{i_1,\ldots,i_k})$, which is also a $G$-equivariant projection of fiber bundle.
\end{itemize}
Then by the natural $G$-equivariant diffeomorphism in Lemma \ref{lem iso-difrnt-rep-of-twistd-sectors}, the above diagram decompose the map $e_{i_1,\ldots, i_k}:X_{[\g]}\rto X_{[\g_{i_1,\ldots,i_k}]}$ into the composition of $\tilde e_{i_1,\ldots,i_k}$ and $p_{i_1,\ldots,i_k}$.

Denote the normal bundle of the embedding $\tilde e_{i_1,\ldots,i_k}$ by $N_{\tilde e_{i_1,\ldots,i_k}}$. Denote the fiber-wise $G$-equivariant volume form of the fibration  $G/C(\g)\rto G/C(\g_{i_1,\ldots,i_k})$ by
\[
\text{vol}(\g_{i_1,\ldots,i_k},\g).
\]
(We will prove the existence of such $G$-equivariant volume form in the appendix, see Theorem \ref{T appendex}.) Then via the vertical map in Diagram \eqref{D} we pull this volume form to the $G$-equivariant fibration $p_{i_1,\ldots,i_k}:X^{\g_{i_1,\ldots,i_k}}\times_{C(\g)} G \rto X^{\g_{i_1,\ldots,i_k}} \times_{C(\g_{i_1,\ldots,i_k})}G$ to get the fiber-wise $G$-equivariant volume form for this fibration. We still denote the pull-back volume form by $\text{vol}(\g_{i_1,\ldots,i_k},\g)$.

All these notations above apply to other maps $e,e_i,e_\0$ in Definition \ref{def natural-maps}.

\begin{defn}\label{D cs-X-product}
For $\alpha_i\in H^*_G(X_{\Conj{g_i}}),i=1,2$, we define the ECS-product of $\alpha_1$ and $\alpha_2$ to be
\begin{align}\label{E def-cs-X-product}
\alpha_1 \star_{cs} \alpha_2
:=&\sum_{
\substack{
\Conj\h=\Conj{h_1,h_2}\in\Conj {G^2_{\sf f}}\\
\Conj{h_i}=\Conj{g_i}, i=1,2}}
p_{\0,*}\left[ \tilde e_{\0,*}\left(e_1^*\alpha_1\wedge
e_2^* \alpha_2\wedge \ee( \scr O_{\Conj \h})\right)\wedge \mathrm{vol}(\h_\0,\h)\right].
\end{align}
The RHS is a finite sum, since a maximal torus of $G$ intersects with every conjugate class $[g]$ in $G$ finitely. The push forward $\tilde e_{\0,*}$ is obtained by wedging equivariant Thom form, and $p_{\0,*}$ is obtained via integration along fiber (cf. \cite{GGK02,GS99}). In particular, $p_{\0,*}(\mathrm{vol}(\h_\0,\h))=1$. Denote the summand on RHS by $(\alpha_1\star_{cs}\alpha_2)_{[\h]}$.
\end{defn}

\begin{remark}
The commutativity of $\g$ ensure the finiteness of the sum in the definition of ``$\star_{cs}$'' and the existence of fiber-wise $G$-equivariant volume form $\mathrm{vol}(\h_\0,\h)$.
\end{remark}

\def \k{{\vec k}}

Our main theorem in this section is
\begin{theorem}\label{T star-asso}
The ECS-product $\star_{cs}$ over $\HH^\ast_{G,cs}(X)$ is associative.
\end{theorem}
\begin{proof}
Take $\alpha_i\in H^*_G(X_{[g_i]})$ for $i=1,2,3$. We next show that
\[
(\alpha_1 \star_{cs} \alpha_2)\star_{cs}\alpha_3=\alpha_1 \star_{cs} (\alpha_2\star_{cs}\alpha_3).
\]
By definition, the LHS is
\begin{align*}
(\alpha_1 \star_{cs} \alpha_2)\star_{cs}\alpha_3
&=\Big(\sum_{
\substack{ \Conj\h=\Conj{h_1,h_2}\in\Conj {G^2_{\sf f}},\\
\Conj{h_i}=\Conj{g_i}, i=1,2}}
\alpha_1\star_{cs}\alpha_2\Big)_{[\h]}\star_{cs}\alpha_3
\\
&=\sum_{
\substack{ \Conj\h=\Conj{h_1,h_2}, \Conj\k=\Conj{k_1,k_2}\in\Conj {G^2_{\sf f}}\\
\Conj{h_i}=\Conj{g_i}, i=1,2,\\
\Conj{k_1}=\Conj{h_1h_2},\Conj{k_2}=\Conj{g_3}}}
((\alpha_1\star_{cs}\alpha_2)_{[\h]} \star_{cs}\alpha_3 )_{[\k]}\\
&=\sum_{
\substack{ \Conj\h=\Conj{h_1,h_2,h_3}\in\Conj {G^3_{\sf f}},\\
\Conj{h_i}=\Conj{g_i}, i=1,2,3}}
((\alpha_1\star_{cs}\alpha_2)_{[\h_{1,2}]} \star_{cs}\alpha_3)_{[\h_{12,3}]}
\end{align*}
where $\h_{1,2}=(h_1,h_2)$ and $\h_{12,3}=(h_1h_2,h_3)$. The third equality follows from the fact that all maximal torus are conjugate and a maximal torus intersects with a conjugate class finitely. Similarly, the RHS is
\begin{align*}
\alpha_1 \star_{cs} (\alpha_2\star_{cs}\alpha_3)
&=\sum_{
\substack{ \Conj\h=\Conj{h_1,h_2,h_3}\in\Conj {G^3_{\sf f}},\\
\Conj{h_i}=\Conj{g_i}, i=1,2,3}}
(\alpha_1\star_{cs}(\alpha_2\star_{cs} \alpha_3)_{[\h_{2,3}]})_{[\h_{1,23}]}
\end{align*}
with $\h_{2,3}=(h_2,h_3)$ and $\h_{1,23}=(h_1,h_2h_3)$.

Now fix an $[\h]\in[G_{\sf f}^3]$ with representative $\h=(h_1,h_2,h_3)$. We compare the contribution $((\alpha_1\star_{cs}\alpha_2)_{[\h_{1,2}]} \star_{cs}\alpha_3)_{[\h_{12,3}]}$ with $(\alpha_1\star_{cs}(\alpha_2\star_{cs} \alpha_3)_{[\h_{2,3}]})_{[\h_{1,23}]}$.
We have the following commutative diagram of maps
\[
\xymatrix{
X^{h_1}\times_{C(h_1)} G &
X^{\h_{1,2}}\times_{C(\h_{1,2})} G
  \ar[l]_-{e_1}
  \ar[dl]_-{e_2}
  \ar@{^{(}->}[r]^-{\tilde e_\0}
  \ar @{} [dr] |{\mathbf{A}}
  &
X^{h_{12}}\times_{C(\h_{1,2})}G
  \ar[r]^-{p_\0}
  \ar @{} [dr] |{\mathbf{B}}&
X^{h_{12}}\times_{C(h_{12})}G
\\
X^{h_2}\times_{C(h_2)} G &
X^{\h_{1,2}}\times_{C(\h)}G
  \ar@{^{(}->}[r]^-{\tilde e_\0} \ar[u]^-{p_{1,2}}
  \ar @{} [dr] |{\mathbf{C}}&
X^{h_{12}}\times_{C(\h)} G
  \ar[r]^-{p_{12,3}} \ar[u]^{p_{1,2}}
  \ar @{} [dr] |{\mathbf{D}}&
X^{h_{12}}\times_{C(\h_{12,3})} G
  \ar[u]^-{p_1}
  \\
&X^\h\times_{C(\h)}G
  \ar@{^{(}->}[u]^-{\tilde e_{1,2}}
  \ar@{^{(}->}[r]^-{\tilde e_{12,3}}
  \ar@{^{(}->}[d]^-{\tilde e_\0}
  &
X^{\h_{12,3}}\times_{C(\h)} G
  \ar[r]^-{p_{12,3}}
  \ar@{^{(}->}[u]^{\tilde e_1}
  \ar@{^{(}->}[dl]^{\tilde e_\0}
  \ar @{} [d] |{\mathbf{E}}
  \ar @{}[dl]_{\mathbf{F}\qq\qq}&
X^{\h_{12,3}}\times_{C(\h_{12,3})}G
  \ar@{^{(}->}[u]^{\tilde e_{1}}
  \ar@{^{(}->}[dl]^-{\tilde e_\0}
  \ar[d]^-{e_2} \\
&X^{\h_\0}\times_{C(\h)} G
  \ar[r]^-{p_{12,3}}\ar[d]^-{p_\0}
  &
X^{\h_\0}\times_{C(\h_{12,3})} G
  \ar[dl]^-{p_\0}
  \ar @{} [dl]_{\mathbf{G}\qq\qq}
&
X^{h_3}\times_{C(h_3)} G   \\
&
X^{\h_\0}\times_{C(\h_\0)} G &&
}\]
Now we compute $((\alpha_1\star_{cs}\alpha_2)_{[\h_{1,2}]} \star_{cs}\alpha_3)_{[\h_{12,3}]}$ and  $(\alpha_1\star_{cs}(\alpha_2\star_{cs} \alpha_3)_{[\h_{2,3}]})_{[\h_{1,23}]}$. First, since $\h_\0=(\h_{12,3})_\0$ we have
\begin{align*}
&((\alpha_1\star_{cs}\alpha_2)_{[\h_{1,2}]} \star_{cs}\alpha_3)_{[\h_{12,3}]}\\
&=p_{\0,*}\left[\tilde e_{\0,*}\left(e_1^*[(\alpha_1\star_{cs}\alpha_2)_{[\h_{1,2}]}]\wedge e_2^*\alpha_3\wedge \ee(\scr O_{\Conj{\h_{12,3}}})
\right)\wedge \text{vol}(\h_\0,\h_{12,3})\right].
\end{align*}
For $e_1^*[(\alpha_1\star_{cs}\alpha_2)_{[\h_{1,2}]}]$, by using the above commutative diagram we get
\begin{eqnarray*}
&&e_1^*[(\alpha_1\star_{cs}\alpha_2)_{[\h_{1,2}]}]
=\tilde e_1^*\circ p^*_1[(\alpha_1\star_{cs}\alpha_2)_{[\h_{1,2}]}]\\
&\stackrel{\mathrm{Definition\, \ref{D cs-X-product}}}{=}&\tilde e_1^*\circ p^*_1\circ p_{\0,*}\left(\tilde e_{\0,*}[e_1^*\alpha_1\wedge e_2^*\alpha_2\wedge\ee(\scr O_{[\h_{1,2}]})]\wedge \text{vol}(h_{12},\h_{1,2})
\right)\\
&\stackrel{\mathbf{B}}{=}&\tilde e_1^*\circ p_{12,3,*}\left(p_{1,2}^*\circ \tilde e_{\0,*}[e_1^*\alpha_1\wedge e_2^*\alpha_2\wedge\ee(\scr O_{[\h_{1,2}]})]\wedge \text{vol}(
\h_{12,3},
\h)\right)\\
&\stackrel{\mathbf{A}}{=}&\tilde e_1^*\circ p_{12,3,*}\left(\tilde e_{\0,*}\circ p_{1,2}^*[e_1^*\alpha_1\wedge e_2^*\alpha_2\wedge\ee(\scr O_{[\h_{1,2}]})]\wedge \text{vol}(\h_{12,3},\h)\right)\\
&\stackrel{\mathbf{D}}{=}& p_{12,3,*}\left(\tilde e_1^*\circ \tilde e_{\0,*}\circ p_{1,2}^*[e_1^*\alpha_1\wedge e_2^*\alpha_2\wedge\ee(\scr O_{[\h_{1,2}]})]\wedge \text{vol}(\h_{12,3},\h
)\right)\\
&\stackrel{\mathbf{C}}{=}& p_{12,3,*}\Big(\tilde e_{12,3,*}\Big(\tilde e_{1,2}^*\circ p_{1,2}^*[e_1^*\alpha_1\wedge e_2^*\alpha_2\wedge\ee(\scr O_{[\h_{1,2}]})
\wedge \ee(E_{12,3})]\Big)
\\
&&\qq\qq\qq\qq\qq\qq\qq  \wedge \text{vol}(\h_{12,3},\h)\Big)\\
&=& p_{12,3,*}\Big(\tilde e_{12,3,*}\Big(e_{1,2}^*[e_1^*\alpha_1\wedge e_2^*\alpha_2\wedge\ee(\scr O_{[\h_{1,2}]})
\wedge \ee(E_{12,3})]\Big)\\
&&\qq\qq\qq\qq\qq\qq\qq
\wedge \text{vol}(\h_{12,3},\h)\Big),
\end{eqnarray*}
where $E_{12,3}$ is the equivariant counterpart of the excess bundle (cf. \cite{Q71}) for the intersection $X^{\h_{12,3}}\cap X^{\h_{1,2}}=X^\h$ in $X^{h_{12}}$, that is
\[
E_{12,3}=\left[(N X^{\h_{1,2}}|X^{h_{12}})|_{X^\h}\ominus NX^\h|X^{\h_{12,3}}\right]\times_{C(\h)} G.
\]
Then we get
\begin{eqnarray*}
&&((\alpha_1\star_{cs}\alpha_2)_{[\h_{1,2}]} \star_{cs}\alpha_3)_{[\h_{12,3}]}\\
&=&p_{\0,*}\Big[\tilde e_{\0,*}\Big\{p_{12,3,*}\Big(\tilde e_{12,3,*}\Big(e_{1,2}^*(e_1^*\alpha_1\wedge e_2^*\alpha_2\wedge\ee(\scr O_{[\h_{1,2}]})\wedge \ee(E_{12,3}))\Big) \\
&&\qq\wedge \text{vol}(\h_{12,3},\h)\Big)\wedge e_2^*\alpha_3\wedge \ee(\scr O_{[\h_{12,3}]})
\Big\}\wedge \text{vol}(\h_\0,\h_{12,3})\Big]\\
&\stackrel{\mathbf{E}}{=}& p_{\0,*}\Big[
p_{12,3,*}
\Big\{\tilde e_{\0,*}
\Big(\tilde e_{12,3,*}
\Big(e_1^*\alpha_1\wedge e_2^*\alpha_2\wedge e_3^*\alpha_3\wedge\ee(\scr O_{[\h_{1,2}]}\Big|_{X_{[\h]}})\\
&&\qq
\wedge \ee(\scr O_{[\h_{12,3}]}\Big|_{X_{[\h]}}) \wedge \ee(E_{12,3}))\Big)\Big)\wedge \text{vol}( \h_{12,3},\h)
\Big\}
\wedge \text{vol}(\h_\0,\h_{12,3})
\Big]\\
&\stackrel{\mathbf{G}}{=}&
p_{\0,*}\Big[
\tilde e_{\0,*}
\Big\{\tilde e_{12,3,*}
\Big(e_1^*\alpha_1\wedge e_2^*\alpha_2\wedge e_3^*\alpha_3\\
&&\qq \wedge
\ee(\scr O_{[\h_{1,2}]}\Big|_{X_{[\h]}}\oplus\scr O_{[\h_{12,3}]}\Big|_{X_{[\h]}}\oplus E_{12,3} )
\Big)\Big\}
\wedge \text{vol}(\h_\0,\h)
\Big]\\
&\stackrel{\mathbf{F}}{=}&
p_{\0,*}\Big[
\tilde e_{\0,*}
\Big(e_1^*\alpha_1\wedge e_2^*\alpha_2\wedge e_3^*\alpha_3\wedge
\ee(\scr O_{[\h_{1,2}]}\Big|_{X_{[\h]}}\oplus \scr O_{[\h_{12,3}]}\Big|_{X_{[\h]}}\oplus E_{12,3})
\Big)\\ &&\qq\qq\qq\qq\qq\qq\qq\qq
\wedge \text{vol}(\h_\0,\h)
\Big].
\end{eqnarray*}
Here we also have used commutative diagrams similar to
\[
\xymatrix{
X_{[\h_{1,2}]}\cong X^{\h_{1,2}}\times_{C(\h_{1,2})} G
\ar[rr]^-{e_i} &&
X^{h_i}\times_{C(h_i)} G\cong X_{[h_i]} \\
X_{[\h]}\cong X_{\h}\times_{C(\h)} G.
\ar[u]^-{e_{1,2}}
\ar[urr]^-{e_i}
}\]
By the same computation we get
\begin{align*}
&(\alpha_1\star_{cs}(\alpha_2 \star_{cs}\alpha_3)_{[\h_{2,3}]})_{[\h_{1,23}]}\\
&=p_{\0,*}\Big[
\tilde e_{\0,*}
\Big(e_1^*\alpha_1\wedge e_2^*\alpha_2\wedge e_3^*\alpha_3\wedge\ee(\scr O_{[\h_{2,3}]}\Big|_{X_{[\h]}}\oplus\scr O_{[\h_{1,23}]}\Big|_{X_{[\h]}}\oplus E_{1,23})
\Big)\\&\qq\qq\qq\qq\qq\qq\qq\qq
\wedge \text{vol}(\h_\0,\h)
\Big]
\end{align*}
with $E_{1,23}$ being the equivariant counterpart of the excess bundle for the intersection $X^{\h_{1,23}}\cap X^{\h_{2,3}}=X^\h$ in $X^{h_{23}}$.

We next compare $\scr O_{[\h_{1,2}]}|_{X_{[\h]}}\oplus \scr O_{[\h_{12,3}]}|_{X_{[\h]}}\oplus E_{12,3}$ with
$\scr O_{[\h_{2,3}]}|_{X_{[\h]}}\oplus \scr O_{[\h_{1,23}]}|_{X_{[\h]}}\oplus E_{1,23}$. We will give a $G$-equivariant $G$-isomorphism between their fibers, which implies a $G$-isomorphism between these two $G$-bundles over $X_{[\h]}$. Consider a point $(x,\h)\in X_{[\h]}$ with $\h=(h_1,h_2,h_3)$. Let $\h_\0=h_1h_2h_3$, $h_0=\h\inv_\0$, $h_{12}=h_1h_2$, $h_{23}=h_2h_3$. Then $\h_{1,2}=(h_1,h_2)$, $\h_{12,3}=(h_{12},h_3)$, $\h_{2,3}=(h_2,h_3)$, $\h_{1,23}=(h_1,h_{23})$. By Lemma \ref{L de-rham-obs}, the fiber of $\scr O_{[\h_{1,2}]}$ over $(x,\h_{1,2})$ is
\[
\mc S_{h_1,x}\oplus \mc S_{h_2,x}\oplus\mc S_{h_{12}\inv,x}\ominus N_{\h_{1,2},x}
\]
and the fiber of $\scr O_{[\h_{12,3}]}$ over $(x,\h_{12,3})$ is
\[
\mc S_{h_{12},x}\oplus \mc S_{h_3,x}\oplus\mc S_{h_0,x}\ominus N_{\h_{12,3},x}
\]
Using the notation in \eqref{E def-N-g-x}, the fiber of $E_{12,3}$ at $(x,\h)$ is
\[
N_{(x,\h)}X^{\h_{1,2}}|X^{h_{12}}\ominus N_{(x,\h)}X^\h|X^{\h_{12,3}}=N_{\h_{1,2},x}\oplus N_{\h_{12,3}}\ominus N_{h_{12},x}\ominus N_{\h,x}.
\]
Then by \eqref{E S+S=N} we see that the fiber of $\scr O_{[\h_{1,2}]}|_{X_{[\h]}}\oplus \scr O_{[\h_{12,3}]}|_{X_{[\h]}}\oplus E_{1,23}$ at $(x,\h)$ is
\begin{align*}
&\mc S_{h_1,x}\oplus \mc S_{h_2,x}\oplus\mc S_{h_{12}\inv,x}\ominus N_{\h_{1,2},x}\oplus \mc S_{h_{12},x}\oplus \mc S_{h_3,x}
\\
&\oplus\mc S_{h_0,x}\ominus N_{\h_{12,3},x}\oplus N_{\h_{1,2},x}\oplus N_{\h_{12,3}}\ominus N_{h_{12},x}\ominus N_{\h,x}\\
=\,&\mc S_{h_1,x}\oplus \mc S_{h_2,x} \oplus\mc S_{h_3,x}\oplus\mc S_{h_0,x} \ominus N_{\h,x}=\scr O_3|_{(x,\h)}.
\end{align*}
Similarly, the fiber of $\scr O_{[\h_{2,3}]}|_{X_{[\h]}}\oplus \scr O_{[\h_{1,23}]}|_{X_{[\h]}}\oplus E_{1,23}$ at $(x,\h)$ is
\begin{align*}
&\mc S_{h_1,x}\oplus \mc S_{h_{23},x}\oplus\mc S_{h_0,x}\ominus N_{\h_{12,3},x}\oplus\mc S_{h_2,x}\oplus \mc S_{h_3,x}\\
&\oplus\mc S_{h_{23}\inv,x}\ominus N_{\h_{2,3},x}\oplus N_{\h_{12,3},x}\oplus N_{\h_{2,3},x}\ominus N_{h_{23},x}\ominus N_{\h,x}\\
=\,&\mc S_{h_1,x}\oplus \mc S_{h_2,x} \oplus\mc S_{h_3,x}\oplus\mc S_{h_0,x} \ominus N_{\h,x}=\scr O_3|_{(x,\h)}.
\end{align*}
Obviously, all these computations are $G$-equivariant. Therefore
\begin{align}\label{E O+O+E=O+O+E}
\scr O_{[\h_{1,2}]}|_{X_{[\h]}}\oplus \scr O_{[\h_{12,3}]}|_{X_{[\h]}}\oplus E_{12,3}
\cong\,&\scr O_{[\h_{2,3}]}|_{X_{[\h]}}\oplus \scr O_{[\h_{1,23}]}|_{X_{[\h]}}\oplus E_{1,23}
,
\end{align}
which implies
\[
((\alpha_1\star_{cs}\alpha_2)_{[\h_{1,2}]} \star_{cs}\alpha_3)_{[\h_{12,3}]}=
(\alpha_1\star_{cs}(\alpha_2\star_{cs} \alpha_3)_{[\h_{2,3}]})_{[\h_{1,23}]}
\]
for any $[\h]\in[G^3_{\sf f}]$ with $[h_i]=[g_i]$. Consequently, the ECS-product $\star_{cs}$ is associative.
\end{proof}

\begin{prop}\label{P preserve-degree}
The ECS-product $\star_{cs}$ preserves the shifted degree. Moreover, it is super-commutative, i.e.
\[
\alpha_1\star_{cs}\alpha_2=(-1)^{\deg \alpha_1 \deg\alpha_2}\alpha_2\star_{cs}\alpha_1.
\]
\end{prop}
\begin{proof}
Using the notation in Definition \ref{D cs-X-product}. Then we need to show that
\begin{align}\label{E evidence-preserve-degree}
\deg (\alpha_1\star_{cs}\alpha_2)_{[\h]}+2\iota([h_1h_2]) =\deg\alpha_1+2\iota([g_1])+\deg\alpha_2+2\iota([g_2]),
\end{align}
for $\h=(h_1,h_2)\in G^2_{\sf f}$ with $[h_i]=[g_i],i=1,2$. By the definition of $\star_{cs}$,
\begin{align*}
\deg(\alpha_1\star_{cs}\alpha_2)_{[\h]} &=\deg\alpha_1+\deg\alpha_2 +\text{rank}\, \scr O_{[\h]}+\text{rank}\, NX^{\h}|X^{h_1h_2}\\
&=\deg\alpha_1+\deg\alpha_2 +\text{rank}\, \scr O_{[\h]}+\text{rank}\, N_{x,\h}-\text{rank}\, N_{x,h_1h_2}.
\end{align*}
By Lemma \ref{L de-rham-obs} we have
\[
\text{rank}\,\scr O_{[\h]}=\text{rank}\,\mc S_{x,(h_1h_2)\inv}+\text{rank}\,\mc S_{x,h_1}+\text{rank}\,\mc S_{x,h_2}-\text{rank}\, N_{x,\h}.
\]
Therefore by \eqref{E S+S=N} we have
\begin{align*}
&\deg(\alpha_1\star_{cs}\alpha_2)_{[\h]}\\
&=\deg\alpha_1+\deg\alpha_2 +\text{rank}\,\mc S_{x,(h_1h_2)\inv}+\text{rank}\,\mc S_{x,h_1}+\text{rank}\,\mc S_{x,h_2}
\\&\qq
-\text{rank}\, N_{x,\h}+\text{rank}\, N_{x,\h}-\text{rank}\, N_{x,h_1h_2}\\
&=\deg\alpha_1+\deg\alpha_2 -\text{rank}\,\mc S_{x,h_1h_2}+\text{rank}\,\mc S_{x,h_1}+\text{rank}\,\mc S_{x,h_2}
\end{align*}
Then \eqref{E evidence-preserve-degree} follows from \eqref{E iota-g=rank-S-g}, i.e. $\iota([g])=\text{rank}_\cplane \mc S_{g,x}$.

The second assertion follows from the fact that the rank of $\scr O_2$ and the degree of $\mathrm{vol}(\h_\0,\h)$ are both even.
\end{proof}

\begin{example}\label{Ex stringy-G=T}
When $G=T$ is a torus, $\HH^*_{T,cs}(X)=\bigoplus_{t\in T_{\sf f}} H^{*-2\iota(t)}(X^t)$, where $T_{\sf f}$ is the subgroup of finite order elements. The diagram \eqref{D} becomes
\[
\xymatrix{
X^\g\ar@{^{(}->}[rr]^-{\tilde e_{i_1,\ldots, i_k}} & &
X^{\g_{i_1,\ldots, i_k}}\ar[rr]^-{p_{i_1,\ldots,i_k}=\text{id}}\ar[d]& & X^{\g_{i_1,\ldots,i_k}}\ar[d]\\
&&\{\text{pt}\} \ar[rr] && \{\text{pt}\}.}
\]
So $e_{i_1,\ldots,i_k}=\tilde e_{i_1,\ldots,i_k}$ and the volume form is trivial. So
\begin{align*}
\alpha_1 \star_{cs} \alpha_2
=e_{\0,*}\left(e_1^*\alpha_1\wedge
e_2^* \alpha_2\wedge \ee( \scr O_{\g})\right)
\end{align*}
for $\alpha_i\in H^*_T(X^{g_i}), i=1,2$, where $\g=(g_1,g_2)$. This product coincides with \cite[Definition 3.3]{GHK07}. The ring $\HH_{T,cs}^*(X)$ is the $\Gamma$-subring of the inertia cohomology (cf. \cite[Definition 6.11]{GHK07}).
\end{example}

\section{Equivariant commutative stringy cohomology ring \\ and symplectic reduction}\label{S 4}

In this section we consider the ECS-cohomology rings for hamiltonian symplectic manifolds and its relation with the Chen--Ruan cohomology of the symplectic reduction orbifolds.

Let $(X,\omega,G,\mu)$ be a hamiltonian system with $G$ being connected and compact. By choosing a $G$-invariant, $\omega$-compatible almost complex structure $J$ we get the ECS-cohomology ring for $(X,G)$
\[
\left(\HH^\ast_{G,cs}(X):=\bigoplus_{[g]\in [G_{\sf f}]}H^{*-2\iota(g)}_G(X_{[g]}), \star_{cs}\right).
\]
This ring does not depend on the choices of $J$, since the space of $G$-invariant, $\omega$-compatible almost complex structures on $X$ is path connected.

Now suppose $0\in\fk g^*$ is a regular value of $\mu$. Denote the level set by $Y:=\mu\inv(0)$. The normal bundle of $Y$ in $X$ is a trivial bundle $Y\times \fk g^*$. By symplectic reduction, there is an induced symplectic form $\omega_{red}$ on the reduction $\msf M:=[Y/G]=X\sslash G$. A $G$-invariant, $\omega$-compatible almost complex structure $J$ on $X$ induces an almost complex structure $\sf J$ on $\sf M$ that is compatible with $\omega_{red}$. Then we get the Chen--Ruan cohomology ring for the symplectic orbifold $\sf M$.

In this section we first construct a ECS-cohomology ring $\HH^*_{G,cs}(Y)$ for the pair $(Y,G)$, and show that the natural inclusion $i:Y\hrto X$ induces a surjective ring homomorphism $i^*:\HH^\ast_{G,cs}(X)\rto \HH^*_{G,cs}(Y)$ in in \S \ref{S 3.1}. Then in \S \ref{S 4.3} we show that we have a natural group isomorphism $\HH^*_{G,cs}(Y)\cong H^*_{CR}(\sf M)$, and in general it is not a ring isomorphism with respect to the ECS-product and the Chen--Ruan product. Moreover, we found that by modifying the hamiltonian system $(X,\omega,G,\mu)$ we could assign the Chen--Ruan cohomology group $H^*_{CR}(\sf M)$ infinite different ring structures that are different from the Chen--Ruan product, and the resulting ring structures are compatible with the ECS-cohomology ring. See Remark \ref{R infinite-ring-struc-on-CR}.

\subsection{ECS-cohomology ring for $(Y,G)$}\label{S 3.1}
Following the Definition \ref{def inertial-mfld} we define the $m$-sectors of $(Y,G)$.
\begin{defn}
For $m\in\integer_{\geq 1}$, we set the $m$-sector of $(Y,G)$ to be  $I_G^m(Y):=\bigsqcup_{[\g]\in[G_{\sf f}^m]} Y_{[g]}$ with
\[
Y_{[\g]}:=\{(y,\g)\in Y\times G_{\sf f}^m|\g\cdot y=(y,\ldots,y)\}.
\]
When $m=1$ we also omit the superscript. $G$ also acts on $Y_{[\g]}$.
\end{defn}

Let $Y^\g=\cap_{i=1}^m Y^{g_i}$ be the fixed loci of $\g=(g_1,\ldots, g_m)$. We also have a $G$-equivariant diffeomorphism $Y_{[\g]}\cong Y^{\g}\times_{C(\g)} G$.

Let $i:Y\hrto X$ being the inclusion map as submanifold. Then one can see that for $m\in\integer_{\geq 1}$, $I_G^m(Y)$ is a $G$-invariant submanifold of $I^m_G(X)$. We denote the inclusion also by $i$. We could restrict those natural maps between $I_G^m(X)$ in Definition \ref{def natural-maps} to $I_G^m(Y)$, and get natural maps between $I_G^m(Y)$. We use the same notations. For example we have the commutative diagram
\begin{align}\label{E i-commute-e}
\begin{split}
\xymatrix{
I_G^2(X) \ar[rr]^-{e_1,e_2,e_\0} && I_G(X) \\
I_G^2(Y) \ar[rr]^-{e_1,e_2,e_\0} \ar@{^{(}->}[u]
&& I_G(Y), \ar@{^{(}->}[u]}
\end{split}
\end{align}
with the lower $e_1,e_2$ and $e_\0$ being the restriction of the upper $e_1,e_2$ and $e_\0$ respectively.

We restrict the degree shifting $\iota([g])$ to $I_G(Y)$ and still denote it by $\iota([g])$.
\begin{defn}\label{D HH(Y,G)}
We define the ECS-cohomology group of $(Y,G)$ as
\[
\HH^*_{G,cs}(Y):=\bigoplus_{[g]\in [G_{\sf f}]} H^{*-2\iota([g])}(Y_{[g]}).
\]
\end{defn}

We restrict the obstruction bundle $\scr O_2$ over $I^2_G(X)$ to $I_G^2(Y)$, and denote it by $\scr O_2^Y$. On the other hand, there is a $G$-equivariant bundle $\fk C_2^*$ over $I_G^2(Y)$ whose fiber at a point $(y,\g)$ is $\fk c_{\g_\0}^*\ominus\fk c_{\g}^*$,
where $\g_\0=g_1g_2$ for $\g=(g_1,g_2)$; $\fk c_{\g_\0}^*$ and $\fk c_{\g}^*$ are the dual of the Lie algebras of the centralizer $C(\g_\0)$ and $C(\g)$ respectively. We set
\[
\tilde {\scr O}_2^Y:=\scr O_2^Y\oplus\fk C^*_2,
\]
and denote the components of $\tilde {\scr O}_2^Y$, $\scr O_2^Y$ and $\fk C_2^*$ over $Y_{[\g]}$ by $\tilde {\scr O}_{[\g]}^Y$, $\scr O^Y_{[\g]}$ and $\fk C_{[\g]}^*$ respectively.

We next define the product over $\HH^*_{G,cs}(Y)$. We also decompose the natural map $e_\0:Y_{[\g]}\rto Y_{[\g_\0]}$ for $\g=(g_1,g_2), $ into
\[
\xymatrix{Y^\g\times_{C(\g)} G\ar@{^{(}->}[rr]^-{\tilde e_\0} &&
Y^{\g_\0}\times_{C(\g)} G\ar[rr]^-{p_\0} && Y^{\g_\0}\times_{C(\g_\0)} G}.
\]
The projection $p_\0: Y^{\g_\0}\times_{C(\g)} G\rto Y^{\g_\0}\times_{C(\g_\0)} G$ is also the pull back of the bundle $p_\0:G/C(\g)\rto G/C(\g_\0)$. Therefore there is a fiberwise $G$-equivariant volume form for $p_\0$, which is also denoted by $\text{vol}(\g_\0,\g)$.
\begin{defn}\label{D cs-Y-product}
For $\alpha_i\in H^*(Y_{[g_i]}), i=1,2$, the ECS-product of them is
\[
\alpha_1\star_{cs}\alpha_2:=\sum_{\substack{\h=[h_1,h_2]\in [G_{\sf f}^2],\\ [g_i]=[h_i],i=1,2}} p_{\0,*}\left[\tilde e_{\0,*}(e_1^*\alpha_1\wedge e^*_2\alpha_2\wedge \ee(\tilde{\scr O}_{[\h]}^Y))\wedge\mathrm{vol}(\h_\0,\h)\right]
\]
\end{defn}

\begin{theorem}
The ECS-product $\star_{cs}$ over $\HH^*_{G,cs}(Y)$ is associative and preserves the shifted degree. Moreover it is supper commutative.
\end{theorem}
\begin{proof}
The proof is similar to the proof of the associativity of $\star_{cs}$ over $\HH^\ast_{G,cs}(X)$ in Theorem \ref{T star-asso} and Proposition \ref{P preserve-degree}. The main part is to prove that over a $3$-sector $Y_{[\h]}$, with $\h=(h_1,h_2,h_3)$, $\h_{1,2}=(h_1,h_2)$, $\h_{12,3}=(h_1h_2,h_3)$, $\h_{2,3}=(h_2,h_3)$ and $\h_{1,23}=(h_1,h_2h_3)$, there is a $G$-equivariant isomorphism
\begin{align}\label{E to-prove-cs-asso-on-Y}
&E_{12,3}^Y
\oplus  \tilde{\scr O}^Y_{\Conj{\h_{1,2}}}\big|_{Y_{[\h]}}
\oplus   \tilde{\scr O}^Y_{\Conj{\h_{12,3}}}\big|_{Y_{[\h]}}
\cong
E_{1,23}^Y
\oplus  \tilde{\scr O}^Y_{\Conj{\h_{2,3}}}\big|_{Y_{[\h]}}
\oplus  \tilde{\scr O}^Y_{\Conj{\h_{1,23}}}\big|_{Y_{[\h]}},
\end{align}
where $E_{12,3}^Y$ and $E_{1,23}^Y$ denote equivariant counterpart of the excess bundles for $Y$, which is similar to the $E_{12,3}$ and $E_{1,23}$ for $X$. For example
\[
E_{12,3}^Y=[(NY^{\h_{1,2}}|Y^{h_{12}})|_{Y^\h} \ominus NY^\h|Y^{\h_{12,3}}]\times_{C(\h)} G.
\]

Since the normal bundle of $Y$ is $X$ is the trivial bundle $Y\times \fk g^*$, one have
\begin{align*}
E_{12,3}^Y=&E_{12,3}|_{Y^\h}
\oplus \fk c_{\h_{1,2}}^*
\oplus\fk c_{\h_{12,3}}^*\ominus\fk c_{h_{12}}^*
\ominus\fk c_{\h_\0}^*,\\
E_{1,23}^Y
=&E_{1,23}|_{Y^\h}
\oplus \fk c_{\h_{2,3}}^*
\oplus\fk c_{\h_{1,23}}^*\ominus\fk c_{h_{23}}^*
\ominus\fk c_{\h_\0}^*,
\end{align*}
where for example, $\fk c_{\h_{1,2}}^*$ is the dual of the Lie algebra of the centralizer $C(\h_{1,2})$ and $h_{12}=h_1h_2=(\h_{1,2})_\0$.

From the computation in the proof of Theorem \ref{T star-asso} and \eqref{E O+O+E=O+O+E}, we see that over a point $(y,\h)\in Y_{[\h]}$, we have an equality of fibers
\begin{eqnarray*}
&&E_{12,3}^Y
\oplus  \tilde{\scr O}^Y_{\Conj{\h_{1,2}}}\big|_{Y_{\Conj\h}}
\oplus  \tilde{\scr O}^Y_{\Conj{\h_{12,3}}}\big|_{Y_{\Conj\h}}\\
&=&E_{12,3}|_{Y^\h}
\oplus \fk c_{\h_{1,2}}^*
\oplus\fk c_{\h_{12,3}}^*\ominus\fk c_{h_{12}}^*
\ominus\fk c_{\h_\0}^*
\oplus
{\scr O}_{\Conj{\h_{1,2}}}\big|_{Y_{\Conj\h}}
\\&&\qq
\oplus \fk C^*_{[\h_{1,2}]} \oplus   {\scr O}_{\Conj{\h_{12,3}}}\big|_{Y_{\Conj\h}}\oplus \fk C^*_{[\h_{12,3}]}
\\
&=&E_{12,3}|_{Y^\h}
\oplus
{\scr O}_{\Conj{\h_{1,2}}}\big|_{Y_{\Conj\h}} \oplus   {\scr O}_{\Conj{\h_{12,3}}}\big|_{Y_{\Conj\h}}
\\
&\stackrel{\eqref{E O+O+E=O+O+E}}{=}&E_{1,23}|_{Y^\h}
\oplus
{\scr O}_{\Conj{\h_{1,23}}}\big|_{Y_{\Conj\h}} \oplus   {\scr O}_{\Conj{\h_{2,3}}}\big|_{Y_{\Conj\h}}
\\
&=&E_{1,23}^Y
\oplus  \tilde{\scr O}^Y_{\Conj{\h_{2,3}}}\big|_{Y_{\Conj\h}}
\oplus   \tilde{\scr O}^Y_{\Conj{\h_{1,23}}}\big|_{Y_{\Conj\h}}.
\end{eqnarray*}
Obviously, this equality is $G$-equivariant. Then we get the $G$-equivariant isomorphism of bundles \eqref{E to-prove-cs-asso-on-Y} over $Y_{[\h]}$. Therefore $\star_{cs}$ is associative on $\HH^*_{cs}(Y,G)$.

The second assertion follows from the same computation in Proposition \ref{P preserve-degree} and the fact that for an $\h\in G^2_{\sf f}$,
\begin{align}\label{E difference-of-normal-bundle}
\begin{split}
&[(NX^{\h}|X^{h_{12}})|_{Y^\h}\ominus NY^{\h}|Y^{h_{12}}]\times_{C(\h)} G\\
&=[(NY^{h_{12}}|X^{h_{12}})|_{Y^\h}\ominus NY^\h|X^\h]\times_{C(\h)} G\\
&=[\fk c^*_{h_{12}}\ominus\fk c^*_{\h}]\times_{C(\h)} G\\
&=\fk C_{[\h]}^*.
\end{split}
\end{align}
In fact, by this equation we get that for $\alpha_i\in H^*_G(Y_{[g_i]})$ we have
\begin{eqnarray*}
&&\deg (\alpha_1\star_{cs}\alpha_2)+\iota([g_1g_2])\\
&\stackrel{\mathrm{Definition}\,\ref{D cs-Y-product}}{=}&\deg\alpha_1+\deg\alpha_2+\mathrm{rank}\,\scr O_2^Y+\iota([g_1g_2])+\mathrm{rank}\, \fk C_{[\g]}^*+\mathrm{rank}\, NY^\g|Y^{\g_\0}\\
&\stackrel{\eqref{E difference-of-normal-bundle}}{=}&\deg\alpha_1+\deg\alpha_2+\mathrm{rank}\,\scr O_2^Y+\iota([g_1g_2])+\mathrm{rank}\, NX^\g|X^{\g_\0}\\
&\stackrel{\mathrm{Proposition}\, \ref{P preserve-degree}}{=}&\deg\alpha_1+\iota([g_1])+\deg\alpha_2+\iota([g_2]).
\end{eqnarray*}
Since the degree of $e_G(\tilde{\scr O}_2^Y)$ and $\mathrm{vol}(\h_\0,\h)$ are both even, the third assertion follows. The proof is accomplished.
\end{proof}

\begin{remark}
From the construction of $(\HH^*_{G,cs}(Y),\star_{cs})$ we see that, we do not need the ambient symplectic manifold $X$ to be compact, but only the level set $Y$ being compact. That is, even if $X$ is noncompact, the construction above works.
\end{remark}

\begin{theorem} \label{T kirwan-map}
The inclusion map $i:I_G(Y)\hrto I_G(X)$ induces a degree preserved, surjective ring homomorphism 
\[
i^*:(\HH^\ast_{G,cs}(X), \star_{cs})\rto (\HH^*_{G,cs}(Y), \star_{cs}).
\]
\end{theorem}
\begin{proof}
By Definition \ref{D HH(X,G)} and Definition \ref{D HH(Y,G)}, wee see that $i^*$ preserves the shifted degree. The rest proof consists of two parts.

\vskip0.1cm
\n{\bf $i^*$ is surjective.}
First note that $i:I_G(Y)\rto I_G(X)$ decomposes into a disjoint union of
\[
i_{[g]}:Y_{[g]}\cong Y^g\times_{C(g)}G\hrto X_{[g]}\cong X^g\times_{C(g)} G,\qq
i_{[g]}([y,h])=[i_g(y),h]
\]
over all $[g]\in [G_{\sf f}]$, where $g\in [g]$ is a representative of the conjugate class $[g]$, and $i_g$ is the inclusion map $i_g:Y^g\hrto X^g$ that embeds $Y^g$ as a $C(g)$-invariant submanifold of $X^g$.
Second, note that we have a commutative diagram
\[
\xymatrix{
H^*_G(Y_{[g]}) \ar[r]^-{i^*_{[g]}} \ar[d]_-\cong &
H^*_G(X_{[g]}) \ar[d]^-\cong \\
H^*_{C(g)}(Y^g)\ar[r]^{i_g^*} &H^*_{C(g)}(X^g).
}
\]
Finally note that $(X^g,\omega)$ is a symplectic submanifold of $(X,\omega)$, and $C(g)$ acts on it in a hamiltonian fashion with moment map being
\[
\mu|_{X^g}:X^g\rto \fk c_g^*\subseteq \fk g^*,
\]
where $\fk c_g^*$, the dual of the Lie algebra of $C(g)$. Then $0$ is also a regular value and $(\mu|_{X^g})\inv(0)=Y^g$. Therefore that the classical Kirwan map is surjective implies that $i^*_g$ is surjective. So is $i^*_{[g]}$.

\vskip0.1cm
\n{\bf $i^*$ is a ring homomorphism.}
Take $\alpha_1\in H^*(X_{[g_1]})$ and $\alpha_2\in H^*(X_{[g_2]})$, and an $\h=(h_1,h_2)\in G_{\sf f}^2$ such that $[h_i]=[g_i]$. Recall that $\h_\0=h_1h_2$. We have the following commutative diagram
\[\begin{tikzpicture}
\def \x{3}
\def \y{-0.7}
\node (A00) at (0,0)     {$X^{\h}$};
\node (A10) at (\x,0)     {$X^{\h_\0}$};
\node (A01) at (0,2*\y)     {$Y^{\h}$};
\node (A11) at (\x,2*\y)     {$Y^{\h_\0}$};
\path (A00) edge [right hook->] (A10);
\path (A01) edge [right hook->] (A11);
\path (A01) edge [right hook->] node [auto] {$\scriptstyle{i}$} (A00);
\path (A11) edge [right hook->] node [auto] {$\scriptstyle{i}$} (A10);
\end{tikzpicture}\]
with all vertical arrows being inclusions.
This gives us the commutative diagram
\[\begin{tikzpicture}
\def \x{3.5}
\def \y{-0.8}
\node (A00) at (0,0)     {$X^{\h}\times_{C(\h)}G$};
\node (A10) at (\x,0)     {$X^{\h_\0}\times_{C(\h)}G$};
\node       at (0.5*\x,\y) {$\textbf{A}$};
\node (A20) at (2*\x,0)     {$X^{\h_\0}\times_{C(h_{12})}G$};
\node (A01) at (0,2*\y)     {$Y^{\h}\times_{C(\h)}G$};
\node (A11) at (\x,2*\y)     {$Y^{\h_\0}\times_{C(\h)}G$};
\node (A21) at (2*\x,2*\y)     {$Y^{\h_\0}\times_{C(h_{12})}G$};
\node       at (1.5*\x,\y) {$\textbf{B}$};
\path (A00) edge [->] node [auto] {$\scriptstyle{\tilde e_\0}$} (A10);
\path (A10) edge [->] node [auto] {$\scriptstyle{p_\0}$} (A20);
\path (A01) edge [->] node [auto] {$\scriptstyle{\tilde e_\0}$} (A11);
\path (A11) edge [->] node [auto] {$\scriptstyle{p_\0}$} (A21);
\path (A01) edge [right hook->] node [auto] {$\scriptstyle{i}$} (A00);
\path (A11) edge [right hook->] node [auto] {$\scriptstyle{i}$} (A10);
\path (A21) edge [right hook->] node [auto] {$\scriptstyle{i}$} (A20);
\end{tikzpicture}\]

Note that we have $i^*\circ p_{\0,*}=p_{\0,*}\circ i^*$ for the second square and
\begin{align*}
i^*\circ \tilde e_{\0,*}(\cdot)=\tilde e_{12,*}\left\{i^*(\cdot)\wedge \ee\left([(NX^{\h}|X^{\h_\0})|_{Y^\h}\ominus NY^{\h}|Y^{\h_\0}]\times_{C(\h)}G\right)\right\}
\end{align*}
for the first square. In fact $(NX^{\h}|X^{\h_\0})|_{Y^\h}\ominus NY^{\h}|Y^{\h_\0}$ is the excess bundle for the intersection of $X^\h\cap Y^{\h_\0}=Y^\h$ in $X^{\h_\0}$. By \eqref{E difference-of-normal-bundle},
\[
[(NX^{\h}|X^{\h_\0})|_{Y^\h}\ominus NY^{\h}|Y^{\h_\0}]\times_{C(\h)} G=\fk C_{[\h]}^*.
\]
Therefore for we have
\begin{eqnarray*}
&&i^*\left[(\alpha_1\star_{cs} \alpha_2)_{[\h]}\right]\\
&\stackrel{\text{Definition }\ref{D cs-X-product}}{=}&i^*\left[p_{\0,*}\Big(\tilde e_{\0,*}(e_1^*\alpha_1\wedge e_2^*\alpha_2
\wedge\ee(\scr O_{\Conj\h}))\wedge\text{vol}(\h_\0,\h )\Big)\right]\\
&\stackrel{\textbf{B}}{=}&p_{\0,*}\left[ i^*\Big(\tilde e_{\0,*}(e_1^*\alpha_1\wedge e_2^*\alpha_2 \wedge\ee(\scr O_{\Conj\h})) \Big)\wedge\text{vol}(\h_\0,\h)\right]\\
&\stackrel{\textbf{A}}{=}&p_{\0,*}\left[\tilde e_{\0,*}\Big(i^*\big(e_1^*\alpha_1\wedge e_2^*\alpha_2 \wedge\ee(\scr O_{\Conj\h})\big)\wedge \ee(\fk C^*_{[\h]}) \Big)\wedge\text{vol}(\h_\0,\h)\right]\\
&=&p_{\0,*}\Big[\tilde e_{\0,*}\Big(i^*\circ e_1^*\alpha_1\wedge i^*\circ e_2^*\alpha_2 \wedge i^*(\ee(\scr O_{\Conj\h}))\wedge \ee(\fk C^*_{[\h]})\Big)
\wedge\text{vol}(\h_\0,\h)\Big]\\
&\stackrel{\eqref{E i-commute-e}}{=}&p_{\0,*}\Big[\tilde e_{\0,*}\Big(e_1^*(i^*\alpha_1)\wedge e_2^*(i^*\alpha_2)\wedge\ee(\scr O^Y_{\Conj\h})\wedge \ee(\fk C^*_{[\h]})\Big)
\wedge\text{vol}(\h_\0,\h)\Big]
\\
&\stackrel{\text{Definition }\ref{D cs-Y-product}}{=}&(i^*\alpha_1 \star_{cs} i^*\alpha_2)_{[\h]}.
\end{eqnarray*}
This implies
$
i^*(\alpha_1\star_{cs} \alpha_2)
=i^*\alpha_1 \star_{cs} i^*\alpha_2.
$ 
So $i^*$ is a ring homomorphism.
\end{proof}

\subsection{Comparing $\HH^*_{G,cs}(Y)$ with $H^*_{CR}(\sf M)$}\label{S 4.3}

The underlying group of $H^*_{CR}(\sf M)$ is the singular cohomology group of its inertia orbifold $\sf IM$. Since $\sf M$ is a quotient orbifold, we have (cf. \cite{ALR07,CH06,CR04})
\[
\msf{IM}=\bigsqcup_{[g]\in[G_{\sf f}]}\msf M_{[g]} =\bigsqcup_{[g]\in[G_{\sf f}]} Y^g/C(g).
\]
So we have
\begin{align*}
H^*_{CR}(\msf M)=\bigoplus_{[g]\in [G_{\sf f}]} H^*(Y^g/C(g)).
\end{align*}
By the definition of ECS-cohomology group, we have
\[
\HH^*_{G,cs}(Y)=\bigoplus_{[g]\in [G_{\sf f}]} H^{*-2\iota([g])}_{G}(Y_{[g]})\cong \bigoplus_{[g]\in [G_{\sf f}]} H^{*-2\iota([g])}_{C(g)}(Y^g)
\]

Note that there is a natural projection map
\begin{align}\label{E pi-on-IGY}
\pi:I_G(Y)\rto \sf IM
\end{align}
by projecting $Y_{[g]}\cong Y^g\times_{C(g)}G$ to $Y^g/C(g)$ for each $[g]\in [G_{\sf f}]$. For each $[g]\in [G_{\sf f}]$, $\pi$ induces a projection over the Borel construction
\begin{align}\label{E pi-on-borel}
\pi:(Y^g\times_{C(g)} G)\times_GEG\rto Y^g/C(g),
\end{align}
which we still denote by $\pi$. Since $0\in \fk g^*$ is regular, the $G$-action on $Y$ is locally free, so is the $C(g)$-action on $Y^g$, hence has finite stabilizers. Then for $g\in G_{\sf f}$ we have a group isomorphism
\[
\pi^*:H^*(Y^g/C(g))\stackrel{\cong}{\rto}H^*_{C(g)}(Y^g),
\]
since the fiber of $\pi$ in \eqref{E pi-on-borel} is rationally acyclic. By summing over $[G_{\sf f}]$ we get a group isomorphism
\begin{align}\label{E cs-cong-CR-group}
\pi^*:H^*_{CR}(\msf M)\rto \HH^\ast_{G,cs}(Y).
\end{align}

In general, for every $m\in\integer_{\geq 1}$ there is a projection map from the $m$-inertia manifold $I^m_G(Y)$ to the $m$-inertia orbifold $\msf I^m\msf M$ of $\sf M$
\[
\pi^m:I^m_G(Y)\rto \msf I^m\msf M=\bigsqcup_{[\g]\in [G^m]}\msf M_{[\g]}=\bigsqcup_{[\g]\in [G^m]} Y^\g/C(\g),
\]
which maps $Y_{[\g]}=Y^\g\times_{C(\g)}G$ to $Y^\g/C(\g)$, where $[G^m]$ is the set of conjugate classes of $m$-tuples of elements\footnote{Here we do not need the assumption on the finiteness of the orders since now the $G$-action on $Y$ is local freely. For an infinite order element $g$, the fixed loci $Y^g$ is empty.} of $G$. When $m=1$, $\pi^1=\pi$ in \eqref{E pi-on-IGY}. Note that generally the map $\pi^m$ is not surjective for $m\geq 2$. The image of $\pi^m$ is
\[
\bigsqcup_{[\g]\in[G_{\sf f}^m]} \msf M_{[\g]}.
\]
There would be $m$-sector $\msf M_{[\g]}$ of $\sf M$ with $\g\notin G_{\sf f}^m$. We denote the image of $\pi^m$ by $\msf I^m_\pi\sf M$. By only using $\msf I^2_\pi\msf M$ we could modify the Chen--Ruan product over $H^*_{CR}(\sf M)$ by setting\footnote{See for example \cite{CR04,ALR07,CH06,HW13} the expression of Chen--Ruan product.}
\[
\alpha_1\,\tilde \cup_{CR}\,\alpha_2=\sum_{[\g]\in[G_{\sf f}^2]}\bar e_{\0,*}(\bar e_1^*\alpha_1\wedge\bar e_2^*\alpha_2\wedge e(\mc O^{CR}_{[\g]}))
\]
for $\alpha_1,\alpha_2\in H^*_{CR}(\sf M)$, where $\bar e_i:Y^{\g}/C(\g)\rto Y^{g_i}/C(\g_i)$ with $i=1,2,\0$, are the evaluation maps, and $\scr O^{CR}_{[\g]}$ is the Chen--Ruan obstruction bundle over the 2-sector $\msf M_{[\g]}$. This is a truncation of the original Chen--Ruan products in \cite{CR04}. Since $[G_{\sf f}^m]$ is closed under multiplication (see the proof of Theorem \ref{T star-asso}) we see that $\tilde\cup_{CR}$ also gives rise to a ring structure over $H^*_{CR}(\sf M)$. We call the resulting ring the {\em commutative Chen--Ruan cohomology ring}, and denote it by $H^*_{CR,cs}(\sf M)$. When $G=T$ is a torus, this is just the  Chen--Ruan cohomology ring.

\begin{prop}\label{P cs-not-eq-cr}
When $G$ is a nonabelian connected compact Lie group, generally the group isomorphism
\[
\pi^*:H^*_{CR,cs}(\msf M)\rto \HH^\ast_{G,cs}(Y)
\]
does not preserves the shifted degree and the products, hence is not a ring isomorphism.
\end{prop}
\begin{proof}
For the first assertion take a $g\in G_{\sf f}$. Without loss of generality, we assume that $Y^g$ is connected. The degree shifting of $\HH^*_{G,cs}(Y)$ associated to $[g]$ is obtained from the $g$-action on $T_xX$ for some $x\in Y^g\subseteq X^g$. The degree shifting of $H^*_{CR}(\sf M)$ associated to $[g]$ is obtained from the $g$-action on $T_{[x]}\sf M$, where $[x]$ means the orbit of $x$ in $Y$. It is well-known that
\begin{align*}
T_xX=T_{[x]}\msf M\oplus \fk g_\cplane,
\end{align*}
with $\fk g_\cplane=\fk g\oplus\fk g^*$. Suppose that $\fk g_\cplane$ decompose into eigen-spaces of $g$-action
\[
\fk g_\cplane=\bigoplus_{0\leq i\leq \ord(g)-1}\fk g_{\cplane,i}
\]
Then the difference between the degree shifting for $\HH^\ast_{G,cs}(Y)$ and the degree shifting for $H^*_{CR}(\sf M)$ associated to $[g]$ is
\[
\sum_{0\leq i\leq \ord(g)-1}\frac{i}{\ord (g)} \dim_\cplane\fk g_{\cplane,i},
\]
which is in general nonzero\footnote{When $G$ is abelian, the adjoint $g$-action on $\fk g_\cplane$ is trivial. Hence the difference is zero.}.

We next consider the second assertion. Take two classes $\alpha_i\in H^*(\msf M_{[g_i]})=H^*(Y^{g_i}/C(g_i))$ for $i=1,2$. Denote by $\beta_i:=\pi^*\alpha_i$, the images in $H^*_{C(g_i)}(Y^{g_i})$. Set $\g=(g_1,g_2)$. Then we have
\[
\alpha_1\,\tilde\cup_{CR}\,\alpha_2 =\sum_{\substack{[\h]=[(h_1,h_2)]\in[G_{\sf f}^2],\\ [h_i]=[g_i], i=1,2}}\bar e_{\0,*}\left(\bar e_1^*\alpha_1\wedge \bar e_2^*\alpha_2\cup e(\scr O^{CR}_{[\h]})\right).
\]
On the other hand
\[
\beta_1\star_{cs}\beta_2=\sum_{\substack{[\h]=[(h_1,h_2)]\in[G_{\sf f}^2],\\ [h_i]=[g_i], i=1,2}} p_{\0,*}\left(\tilde e_{\0,*}(e_1^*\beta_1\wedge e_2^*\beta_2\wedge e_G(\tilde{\scr O}^Y_{[\h]}))\wedge\mathrm{vol}(\h_\0,\h)\right).
\]
We next show that $\beta_1\star_{cs}\beta_2\neq \pi^*(\alpha_1\,\tilde\cup_{CR}\,\alpha_2)$ generally. We compare their components for every $[\h]=[(h_1,h_2)]\in[G_{\sf f}^2]$ satisfying $[h_i]=[g_i], i=1,2$.

For an $\h\in G_{\sf f}^2$ we have a commutative diagram
\begin{align}\label{D CR-CS-proof}\begin{split}
\xymatrix{Y_{[\h]}=(Y^\h\times_{C(\h)} G) \ar[d]_-\pi \ar[r]^-{\tilde e_\0} &
Y^{\h_\0}\times_{C(\h)} G  \ar[r]^-{p_\0} \ar[dr]_-p&
Y^{\h_\0}\times_{C(\h_\0)} G=Y_{[\h_\0]}  \ar[d]^-\pi\\
\msf M_{[\h]}=Y^\h/C(\h)   \ar[rr]^-{\bar e_\0} &
& \msf M_{[\h_\0]}=Y^{\h_\0}/C(\h_\0).}
\end{split}
\end{align}
Denote the normal bundle of $\bar e_\0$ by $N_{\bar e_\0}$, and the normal bundle of $\tilde e_\0$ by $N_{\tilde e_\0}$ for simplicity. We have seen that over $Y_{[\h]}$ there is a bundle $\fk C^*_{[\h]}$. Denote its dual bundle by $\fk C_{[\h]}$. Then we have
\[
\pi^*N_{\bar e_\0}\oplus \fk C_{[\h]}=N_{\tilde e_\0}.
\]
We also pull back the Chen--Ruan obstruction bundle to $Y_{[\h]}$. Take a point $x\in Y^\h$. Then we get a point $[x]\in Y^\h/C(\h)\subseteq \msf I^2\sf M$ via the natural projection $\pi:I^2_G(Y)\rto \msf I^2\sf M$. The fiber of $\tilde{\scr O}^Y_{[\h]}$ at $(x,\h)$ is
\[
\fk C^*_{[\h]}|_{(x,\h)}\oplus\bigoplus_{\lambda\in \widehat{\langle\h\rangle},\;
\sum_{i=0}^m w_{\lambda,i}\geq 2}
\Big(\sum_{i=0}^m w_{\lambda,i}-1\Big)\cdot T_{x,\lambda}
\]
where $\oplus_{\lambda\in \widehat{\langle\h\rangle}} T_{x,\lambda}$ is the irreducible decomposition of $T_xX$ under the $\langle\h\rangle$-action. Recall that $T_xX=T_{[x]}\msf M\oplus\fk g_\cplane$, and by the computation in \cite[Theorem 3.2]{HW13}, the obstruction bundle $\scr O^{CR}_{[\h]}$ has fiber over $[x]$ being
\[
\bigoplus_{\lambda\in \widehat{\langle\h\rangle},\;
\sum_{i=0}^m w_{\lambda,i}\geq 2}
\Big(\sum_{i=0}^m w_{\lambda,i}-1\Big)\cdot \msf T_{x,\lambda}
\]
where $\oplus_{\lambda\in \widehat{\langle\h\rangle}} \msf T_{x,\lambda}$ is the irreducible decomposition of $T_{[x]}\sf M$ under the $\langle\h\rangle$-action. One see that
\begin{align*}
\fk C^*_{[\h]}|_{(x,\h)}+ \sum_{\lambda\in \widehat{\langle\h\rangle},\;
\sum_{i=0}^m w_{\lambda,i}\geq 2}
\Big(\sum_{i=0}^m w_{\lambda,i}-1\Big)\cdot \fk g_{\cplane,\lambda}
\end{align*}
forms a $G$-bundle over $Y_{[\h]}$, where $\oplus_{\lambda\in \widehat{ \langle\h\rangle}} \fk g_{\cplane,\lambda}$ is the irreducible decomposition of $\fk g_\cplane$ under the $\langle\h\rangle$-action. We denote this bundle by $V_{[\h]}$. Then over $Y_{[\h]}$ we have
\[
\tilde{\scr O}^Y_{[\h]}=\pi^*\scr O^{CR}_{[\h]}\oplus V_{[\h]}.
\]
Therefore
\begin{align*}
(\beta_1\star_{cs}\beta_2)_{[\h]}= p_{\0,*}\Big(& \big[ e_1^*\pi^*\alpha_1\wedge e_2^*\pi^*\alpha_2 \wedge \Theta_G(\pi^*N_{\bar e_\0})\wedge  \Theta_G(\fk C_{[\h]})\\
&\wedge e_G(V_{[\h]})\wedge e_G(\pi^*\scr O^{CR}_{[\h]})\big]\wedge \mathrm{vol} (\h_\0,\h)\Big),
\end{align*}
On the other hand, by the triangle in the commutative diagram \eqref{D CR-CS-proof} we have
\begin{align*}
\pi^*(\alpha_1\,\tilde\cup_{CR}\,\alpha_2)_{[\h]}&=
p_{\0,*}\left(p^*\left[\bar e_{\0,*}(\bar e_1^*\alpha_1\wedge \bar e_2^*\alpha_2\cup e(\scr O^{CR}_{[\h]}))\right] \wedge\mathrm{vol}(\h_\0,\h)\right).
\end{align*}
It is direct to see that $e_1^*\pi^*\alpha_1\wedge e_2^*\pi^*\alpha_2 \wedge \Theta_G(\pi^*N_{\bar e_\0})\wedge e_G(\pi^*\scr O^{CR}_{[\h]})$ corresponds to the $p^*\left[\bar e_{\0,*}(\bar e_1^*\alpha_1\wedge \bar e_2^*\alpha_2\cup e(\scr O^{CR}_{[\h]}))\right]$. Therefore, the difference between $\pi^*(\alpha_1\,\tilde \cup_{CR}\,\alpha_2)_{[\h]}$ and $(\beta_1\star_{cs}\beta_2)_{[\h]}$ is determined by $\Theta_G(\fk C_{[\h]})\wedge e_G(V_{[\h]})$. When $G$ is non-abelian, these two bundles have non-zero equivariant characteristics generally. Therefore $i^*$ is not a ring homomorphism in general.
\end{proof}

\begin{example}
When $G=T$ is a torus,
\begin{itemize}
\item $\pi^*$ preserves the shifted degree, since the $T$-action on its Lie algebra is trivial;
\item $p_\0=id$ and $e_\0=\tilde e_\0$ for the decomposition $e_\0=p_\0\circ\tilde e_\0$ (cf. \eqref{D});
\item $\msf I^m\msf M=\msf I^m_\pi\msf M$ and $\tilde\cup_{CR}$ coincides with the original Chen--Ruan product $\cup_{CR}$;
\item the commutative diagram \eqref{D CR-CS-proof} reduces to
\[
\xymatrix{Y^\h\times_T T\ar[d]_-\pi \ar[r]^-{ e_\0}&
Y^{\h_\0}\times_T T \ar[r]^-{id} \ar[dr]_-{p=\pi} &
Y^{\h_\0}\times_T T\ar[d]^-\pi \\
Y^\h/T   \ar[rr]^-{\bar e_\0} &&
Y^{\h_\0}/T.}
\]
\end{itemize}
Then we see that both $\fk C_{[\h]}$ and $V_{[\h]}$ are zero bundle. Therefore, for this case $\pi^*$ is a ring isomorphism between $\HH^*_{G,cs}(Y)$ and the Chen--Ruan cohomology ring $H^*_{CR}(\sf M)$. In fact, for this case, $Y$ is a $T$-equivariant stable almost complex manifold, and $\HH^*_{G,cs}(Y)$ is the $\Gamma$-subring of the inertia cohomology. Then by combining with Theorem \ref{T kirwan-map} we recover \cite[Corollary 6.12]{GHK07}.
\end{example}

\begin{remark}\label{R infinite-ring-struc-on-CR}
Via the group isomorphism \eqref{E cs-cong-CR-group}, we  could transfer the product $\star_{cs}$ to $H^*_{CR}(\sf M)$. With this new ring structure, we get a surjective ring homomorphism
\[
\xymatrix{\HH^\ast_{G,cs}(X)\ar[r]^-{i^*} &
\HH^*_{G,cs}(Y)\ar[r]^-\cong  & \red{
(H^*_{CR}(\msf M),\star_{cs})}.}
\]
We view this as a Kirwan morphism for ECS-cohomology ring.

It is obvious that, different hamiltonian systems would have the same symplectic reduction orbifold. For example, Let $(X_i,\omega_i,G_i,\mu_i)$ be two hamiltonian systems for $i=1,2$, and $0\in\fk g_i^*$ be regular values of $\mu_i$ for $i=1,2$. Suppose that the reduction symplectic orbifolds satisfy
\[
[\mu_1\inv(0)/G_1]\cong [\mu_2\inv(0)/G_2]\cong \sf M.
\]
Then we get two ring structures over the Chen--Ruan cohomology group $H^*_{CR}(\sf M)$ via the group isomorphisms
\[
\HH^*_{G_i,cs}(\mu_i\inv(0))\cong H^*_{CR}(\msf M),\qq\mathrm{for}\qq i=1,2.
\]
Then by the same computations in the proof of Proposition \ref{P cs-not-eq-cr} we see that these two induced ring structure are not the same in general.

So, if a symplectic orbifold $\sf M$ is a symplectic reduction of a hamiltonian system $(X,\omega,G,\mu)$, we could get infinite ring structures over its Chen--Ruan cohomology group by simply enlarging the hamiltonian system $(X,\omega,G,\mu)$ into
\[
(X\times T^*H,\omega\oplus d\lambda, G\times H, \mu\oplus\mu_H)
\]
for every connected compact Lie group $H$, where $(T^*H,d\lambda,H,\mu_H)$ is the canonical hamiltonian system associated to the cotangent bundle of $H$, and $\lambda$ is the Liouville form (cf. \cite{Can08}).
\end{remark}
\appendix

\section{Existence of equivariant volume form}\label{S appendix}
In this appendix we show the existence of the equivariant volume form that we used in the definition of commutative stringy product. Let $G$ be a connected compact Lie group and $T$ be one of its maximal torus. Let $\fk g$ and $\fk t$ be their Lie algebra. $G$ acts on $G$ by conjugation and on $\fk g$ by adjoint representation.

\begin{prop} \label{P appendex}
For any $h\in T$, there exists $\alpha\in\fk t$ such that
$C(h)=C(\alpha)$.
\end{prop}
\begin{coro}\label{C appendex}
$G/C(h)$ is a K\"ahler manifold. Moreover, it has a $G$-equivariant volume form $\Omega^G_h$.
\end{coro}
\begin{proof}
Since $G/C(\alpha)$ is a (co)-adjoint orbit, it is K\"ahler and the $G$-action on it is Hamiltonian. Denote its K\"ahler form by $\omega_\alpha$ and $\omega_\alpha+ \mu_\alpha$ be its equivariant extension, where $\mu_\alpha$ is the moment map for the Hamiltonian $G$-action on $G/C(\alpha)$. Then $(\omega_\alpha+\mu_\alpha)^d$ is an equivariant volume form, where $d$ is the complex dimension of $G/C(\alpha)$.
\end{proof}

\n{\em Proof of Proposition \ref{P appendex}}.
Since $T$ is abelian, its adjoint action on $\fk g$ induces a
splitting
\[
\fk g = \fk t \oplus \cplane_1\oplus\cdots\oplus\cplane_k
\]
and the action is given by $1\oplus\phi_1\oplus\cdots\oplus\phi_k$. Let $T_i=\{t\in T\mid \phi_i(t)=1\}$, and for any $I\subseteq\{1,\ldots,k\}$, set
\[
T_I=\bigcap_{i\in I}T_i
\]
and $T_\varnothing =T$. Then $T$ is stratified by
\[
T'_I=T_I\setminus \bigcup_{J\supseteq I} T_J.
\]
Similarly, let $\fk t_I$ be the Lie algebra of $T_I$ and $\fk t'_I$ forms a stratification of $\fk t$.

Suppose that $h\in T'_I$, then choose an $\alpha\in \fk t'_I$. It is sufficient to show that $\fk c(h)=\fk c(\alpha)$ which implies that $C(h)=C(\alpha)$, since both $C(h)$ and $C(\alpha)$ are connected subgroup of $G$ (cf. \cite[Corollary 2, \S 3.1]{H04}). In fact,
\[
\fk c(h)=\{\xi\mid h\xi h\inv=\xi\}
\]
Since $h\in T'_I$, it fixes $\fk t$ and all $\cplane_i$ for $i\in I$. Hence $\fk c(h)=\fk t\oplus \bigoplus_{i\in I}\cplane_i$. Similarly, $\fk c(\alpha)$ is the same space. \qed

\begin{theorem}\label{T appendex}
Suppose that $h_1, h_2\in T$ and $\h=(h_1,h_2)$. The fibration $p_1:G/C(h_1, h_2)\rto G/C(h_1)$ admits a fiber-wise $G$-equivariant volume form $\mathrm{vol}(\h,h_1)$, so $p_{1,*}(\mathrm{vol}(\h,h_1))=1$.
\end{theorem}
\begin{proof}
The fiber of $p_1:G/C(h_1,h_2)\rto G/C(h_1)$ is  $C(h_1)/C(h_1,h_2)$. Set $K=C(h_1)$. Since $h_1$ and $h_2$ commutes, $h_2\in K$. Hence $C(h_1,h_2)=C_K(h_2)$ and
\[
C(h_1)/C(h_1,h_2)=K/C_K(h_2).
\]
By Corollary \ref{C appendex}, there exists a $K$-equivariant volume form $\Omega^K_{h_2}$ on $K/C_K(h_2)$. Note that
\[
G/C(h_1,h_2)=K/C_K(h_2)\times_K G.
\]
For each point $[h]\in G/C(h_1)$, write the fiber over $[h]$ by $F_{[h]}$. Then by setting $\mathrm{vol}(\h,h_1)|_{F_{[h]}}=h^*\Omega^K_{h_2}$ for each point $[h]\in G/C(h_1)$ we get this volume form $\mathrm{vol}(\h,h_1)$. We can also get this volume form $\mathrm{vol}(\h,h_1)$ via the canonical isomorphism
\[
H^*_G(G/C(h_1,h_2))=H^*_G(K/C_K(h_2)\times_K G)\cong H^*_K(K/C_K(h_2)).
\]
\end{proof}
By induction on the length of $\h=(h_1,\ldots,h_n)$, this theorem could be generalized to $G/C(\h)\rto G/C(\h_{i_1,\ldots, i_k})$ and $G/C(\h)\rto G/C(\h_\0)$. Therefore all kinds of fibrations used in the definition of equivariant commutative stringy products have fiber-wise $G$-equivariant volume forms.

\end{document}